\documentclass[12pt,reqno]{amsart}

\usepackage[T2A]{fontenc}
\usepackage[cp866]{inputenc}
\usepackage[russian,english]{babel}
\usepackage{amsfonts,amssymb}
\usepackage{mathrsfs}
\usepackage[all]{xy}
\usepackage{hypbmsec}
\usepackage{mathtext}

\usepackage{amsbib}

\theoremstyle{plain}
\newtheorem{theorem}{Theorem}
\newtheorem{proposition}{Proposition}
\newtheorem{statement}{Assertion}
\newtheorem{corollary}{Corollary}
\newtheorem{lemma}{Lemma}

\theoremstyle{remark}

\newtheorem{remark}{Remark}

\theoremstyle{definition}
\newtheorem{definition}{Definition}

\begin{document}

\def\refname{Bibliography}

\translator{A. I. Shtern}


\title{Universal adic approximation, invariant measures and scaled entropy}

\author{A.~M.~Vershik}
\address{St.~Petersburg Department of the Steklov 
\\
Mathematical Institute of the Russian Academy 
\\
of Sciences, St.~Petersburg
\\[0.7pt]
Department of Mathematics and Mechanics, 
\\
St.~Petersburg State University, St.~Petersburg
\\[0.7pt] 
Institute for Information Transmission Problems 
\\
of~the~Russian Academy of Sciences 
\\
(Kharkevich~Institute), Moscow}
\email{avershik@gmail.com}

\author{P.~B.~Zatitskii}
\address{D\'epartement de math\'ematiques et applications, 
\\
\'Ecole Normale Sup\'erieure, Paris
\\[0.7pt]
Department of Mathematics and Mechanics, 
\\
Chebyshev Laboratory, St.~Petersburg State 
\\
University, St.~Petersburg
\\[0.7pt]
St.~Petersburg Department of the Steklov 
\\
Mathematical Institute of the Russian Academy 
\\
of Sciences, St.~Petersburg} 
\email{pavelz@pdmi.ras.ru}


\maketitle

\markright{Universal adic approximation, invariant measures
and scaled entropy}

\thanks{This research was financially supported
by the Russian Science Foundation
under grant no.~14-11-00581).}

\begin{abstract}
We define an infinite graded graph of ordered pairs and a~canonical action
of the group $\mathbb{Z}$ (the adic action) and of the infinite sum of groups of order
two~$\mathcal{D}=\sum_1^{\infty} \mathbb{Z}/2\mathbb{Z}$ on the path space
of the graph. It is proved that these actions are universal for both groups
in the following sense: every ergodic action of these groups with invariant
measure and binomial generator, multiplied by a~special action (the `odometer'),
is metrically isomorphic to the canonical adic action on the path space of the
graph with a~central measure. We consider a~series of related 
problems.

\end{abstract}


\section{Introduction}
\label{s1}

\vskip0.7pt

This paper is devoted to the solution of several interrelated problems:

\vskip0.7pt

-- we prove the existence of a~universal adic realization of an arbitrary
ergodic action of the group~$\mathbb{Z}$ and the group
$\mathcal{D}=\sum_1^{\infty} \mathbb{Z}/2\mathbb{Z}$, which is an infinite
direct sum of groups of order two, on the path space of the \textit{graded
graph (Bratteli diagram) of ordered pairs}, which is denoted below
by~$\mathrm{OP}$ (Ordered Pairs);

\vskip0.7pt

-- we list all ergodic central measures on the path space of the
graph~$\mathrm{OP}$;

\vskip0.7pt

-- we prove that the scaling sequence for a~dyadic filtration 
does not depend
on the choice of the iterated metric;

\vskip0.7pt

-- we prove that for the adic actions of the groups~$\mathbb{Z}$ and~$\mathcal{D}$
on the path space of~the~graph~$\mathrm{OP}$ there are invariant measures
having a~given subadditive growth~of scaling sequences in the definition
of scaled entropy, and here the entropy of the tail filtration with respect
to this measure has the same growth.

\vskip1pt

Let us comment on these problems and their solutions.

\subsection{Universal adic realization}
\label{ss1.1}
\looseness=1
In the classical lemma of Rokhlin, a~periodic approximation of arbitrary order is
constructed for any aperiodic automorphism with invariant measure. It is
possible to construct a~coherent system of~`Rokhlin towers' defining an
exhaustive periodic approximation (for example, of orders~$2^n$, $n\in \mathbb N$): 
in~\cite{1} and~\cite{2}, the so-called `adic realization' of any
ergodic action was constructed. Namely, it was proved that for every ergodic
automorphism~$S$ there is a~graded graph~$\Gamma$, an adic structure
on~$\Gamma$, and a~central measure~$\nu$ on the path
space~$\mathcal{T}(\Gamma)$ of~$\Gamma$ such that the adic shift
on~$\mathcal{T}(\Gamma)$ is isomorphic to~$S$. The following
question arises: is it possible to refine this construction and implement
arbitrary ergodic actions on the path space \textit{of the same graph} by
modifying only the invariant measure? One can make the question more specific:
consider the so-called symbolic action on the space~$2^G$ of a~discrete 
group~$G$ by (left) shifts; this is a~universal model: any action with
a~binomial generator can be realized in~this way. Then our question, even in the case
when $G=\nobreak\mathbb{Z}$, actually reduces to the following: can we provide 
the construction of a~sequence of approximations of a~`Rokhlin tower' with a~universal 
nature, that is, make it into the adic approximation? However, a~positive answer 
to this question suggests a~positive answer to another, more modest, question: 
is it possible to give a~universal (if possible, constructive) proof of Rokhlin's
lemma for an arbitrary homeomorphism, for example, for a~shift in the space
$2^{\mathbb{Z}}$, simultaneously for all invariant Borel probability measures?
Approximately this question was posed to the first author 
by~Rokhlin. The point is that an uncountable induction was used in all proofs
known so far, and it is desirable to avoid this induction without using
any properties of a~particular measure. It would then be possible
to raise the same question about the adic approximation.\footnote{
(\textit{added in proof}). Already submitting this paper, the
authors learned from a~conversation with~B.~Weiss that such a~proof is
given in the book \textit{`Handbook of Dynamical Systems'}, 2006. In the
chapter `\textit{On the interplay between measurable and topological
dynamics\/}', the authors, Glasner and~Weiss, give a~universal construction
(with respect to any invariant Borel measure) of a~periodic approximation for
an arbitrary aperiodic homeomorphism of a~Polish space.}
It is shown in this paper that such a~universal construction is possible if the question is
modified as follows: \textit{approximate (or construct an adic realization of)
a~direct product of an arbitrary action by the} `\textit{odometer\/}'
\textit{rather than by an arbitrary action of\/~$\mathbb{Z}$}. 
By the odometer we mean the operation of adding one to any element of the
group~$\textbf{Z}_2$ of dyadic numbers; this is an ergodic automorphism with
a~dyadic spectrum. In the group $\mathcal{D}=\sum_1^{\infty}
\mathbb{Z}/2\mathbb{Z}$ considered below, the role of the odometer is played by the
action of~$\mathcal{D}$ on the group $(\mathbb{Z}/2\mathbb{Z})^\mathbb{N} =
\prod_{j=1}^{\infty} (\mathbb{Z}/2\mathbb{Z})$, which is the group
of characters of~$\mathcal{D}$. Thus, the universality of the approximation is
readily achieved by multiplying directly by an action of the simplest kind. As
a~universal space, the direct product of the space~$2^\mathcal{D}$ and the
odometer space arises, which, in our case, coincides with the path space of the
graph of ordered pairs. Thus, a~universal dyadic approximation is constructed.
It would be interesting to find out whether or not such a~possibility exists for
non-Abelian amenable groups.\footnote{
The general problem is as follows: suppose
that a~hyperfinite equivalence relation (that is, the trajectory partition
of a~Borel action of the group $\mathbb Z$; see~\cite{3}) is defined on the standard
Borel space or, in particular, on a~separable metric space. Then find an effective
Borel isomorphism between this space and the path space of some graded locally
finite graph that maps the tail equivalence relation on the path space of the
graph into a~given equivalence relation. It suffices to solve this problem for
the space $2^{\mathbb Z}$ and the usual action of the group $\mathbb Z$
by shifts on this space. It is not known for what amenable groups this
construction is possible in principle.}

\subsection{Invariant central measures}
\label{ss1.2}
Finding invariant measures for group actions on a~given compact or topological  
space is a~traditional problem of the theory of dynamical
systems. In recent papers it has been shown that, in many cases,
this problem can be reduced to the description of the so-called central
measures on the path spaces of graded graphs (or Bratteli diagrams). In our
case, we speak of a~specific graph, namely, the graph $\mathrm{OP}$ of ordered
pairs. Its construction (see~\S\,\ref{s2}) is very simple: the set of vertices
on the next floor is the set of all ordered pairs of vertices on the previous
floor. This graph was studied in~\cite{4}, as well as a~more complicated
graph of unordered pairs (a~`tower of measures').

There is a~natural bijection between the set of central measures on the path
space of the graph of ordered pairs and the set of measures on the space
$2^{\mathbb Z}\times [0,1]$ that are invariant under the direct
product of the actions of the group~$\mathbb Z$ (the shift on the first component
and the odometer on the segment). Along with the direct products of invariant
measures (the invariant measure is unique for the odometer), there is a~series
of the most interesting ergodic measures with respect to which this action is
a~skew product. We describe them completely. Moreover, it is possible
to describe completely the types of actions of the group~$\mathbb Z$ with respect
to these measures: they are the actions that have a~factor isomorphic to the
odometer. We note that the necessity of this condition is trivial, while the
sufficiency is not so obvious. A~similar result holds for the
group~$\mathcal{D}$.

The case under consideration is part of the following general problem. When
there are two continuous actions of a~group~$G$ on two compact spaces,  
$X$ and~$Y$,
describe all measures invariant under the direct product of actions of~$G$
on the product $X\times\nobreak Y$ modulo the invariant measures on~$X$ and~$Y$
(here we also know in addition that an invariant measure is unique on one of the
spaces, for example, on~$X$). However, even under this assumption, the problem
looks immense since the answer must include descriptions of all quotient
actions of~$G$ on both spaces. In our case, we used the fact that all
quotient actions of the odometer are well known, and therefore the problem
of invariant measures has a~clear answer.

\subsection{Scaled entropy of actions and filtrations}
\label{ss1.3}
The scaled entropy of dynamical systems was defined and studied
in recent papers of the first author (\cite{5},~\cite{6}). The theory
of filtrations (decreasing sequences of sigma-algebras) was a~motivation for the
definition: scaled entropy was first defined as an invariant
of a~filtration (see~\cite{7}--\cite{10}). Below we establish the numerical
coincidence of these notions in our situation. However, the importance of the
notion of scaled entropy of dynamical systems has forced us to give an 
interpretation in terms of an adic implementation of the action, which 
is done in this paper. The novelty of the notion is that, although it is a~purely
metric (rather than topological) invariant of the action, it is defined using
the $\varepsilon$-entropy of metric spaces rather than the metric entropy
of partitions. The scaling sequence is first defined for an automorphism
of a~measure space equipped additionally with a~metric (that is, an admissible
triple: cf., for example,~\cite{11}), as a~sequence that normalizes the
$\varepsilon$-entropy of the space equipped with an invariant measure and an
iterated metric. A~conjecture of the first author, which was proved in the
thesis of the second (see~\cite{12} and~\cite{13}), is that the
asymptotic behaviour of this sequence does not depend on the choice of the
original metric. Thus, the scaled entropy is a~new metric invariant of an action
of groups, which considerably generalizes Kolmogorov's entropy theory.

As mentioned above, the notion of scaled entropy arose from an analysis
of the preceding notion, namely, the \textit{entropy of filtration}, which was
suggested as a~metric invariant of filtration, that is, of a~decreasing
sequence of sigma-algebras (or measurable partitions). In this paper, we dwell
on the theory of filtrations only within the limits in which it is necessary
for the main topic of the paper, the more so because this theory, in its
modern version, will be described in detail in a~forthcoming paper
of the first author. In one of the special cases (homogeneous partitions), the
entropy of a~filtration was defined without using any metric on a~measure
space. However, in the general case, it is convenient to define this entropy
(as well as the scaled entropy of an action) first in a~space with a~metric (or
semimetric) and then prove the independence of the choice of the original
metric. Therefore, in this paper we present an argument about the independence
analogous to that used for the scaling sequence of an action. Namely, we show
that, in the definition of a~scaling sequence of a~filtration, it suffices
to restrict ourselves to any admissible metric without taking the supremum over
all admissible semimetrics.

In various special cases, ideas similar to that of the scaled entropy of an action
have already been used by various authors (see~\cite{14} for the sequential
entropy or the Kirillov--\allowbreak Kushnirenko entropy, \cite{15} and~\cite{16} for 
loosely Bernoulli property, \cite{17}~for~slow entropy, and \cite{18} for complexity
entropy). One of the first applications of the idea of scaled entropy is that
the boundedness of a~scaling sequence is equivalent to the condition
of a~purely point spectrum (see~\cite{14},~\cite{18},~\cite{11}). The definition
of scaled entropy seems to be the most general and synthetic, uniting the
concepts of metric and $\varepsilon$-entropy and revealing the auxiliary
(non-principal) role of the metric.

\looseness=-1
A~result outlined in the paper \cite{19} of Ferenczi and~Park means that the
scaling sequence for ergodic automorphisms can have any intermediate asymptotic
behaviour, and this was proved rigorously in~\cite{13}. In \cite{13},  
examples
of special central measures on the path space of the graph~$\mathrm{OP}$ were
constructed for which the adic transformation has an arbitrary prescribed
subadditive growth of the scaling sequence, and in~\cite{20} it was shown that
a~subadditive sequence must exist in any non-empty class of scaling sequences. In
the above examples with intermediate growth, a~free action of the
group~$\mathcal{D}$ is constructed using a~non-free action on~$2^\mathcal{D}$
for which a~measure is concentrated on the set of functions that are constant
on the cosets of some subgroup of~$\mathcal{D}$, which, in turn, is chosen for
the given scaling sequence. In a~more cumbersome way, this effect
of intermediate growth can also be explained for the group~$\mathbb{Z}$ as an
effect of an action on some class of functions 
(that~is, on~$2^\mathbb{Z}$) with hidden symmetries, but a~good formulation 
has yet to be found.

The result on the arbitrariness of the asymptotic behaviour of a~scaling
sequence for the group $\mathbb{Z}$ was proved in~\cite{20} using the adic
implementation of the action. Here it is also proved for locally finite Abelian
groups. For some special measures this asymptotic behaviour for the adic
transformation coincides with the asymptotic behaviour of the entropy of the
tail filtration of the path space of the graph of ordered pairs. This
coincidence is possibly of a~more general nature and not connected with 
specific features of the graph and the measures considered on the path space.

Certainly, our consideration of the graph of ordered pairs and dyadic groups can
readily be generalized, for example to the graph of ordered triples
or~$n$-tuples, with the dyadic group replaced by the triadic, and so~on. 
This introduces no fundamental changes into the results, only
the numerical characteristics being modified. Moreover, we can consider the
$r_n$-adic case in which the number of joined points (instead of pairs) varies
from floor to floor, but remains permanent on each floor. This~is~the case
of the so-called $\{r_n\}$-adic filtrations (see~\cite{9}). The groups
corresponding~to the group~$\mathcal{D}$ are direct sums of the corresponding
groups $\mathbb Z/r_n \mathbb{Z}$. The odometer is replaced by an automorphism
whose discrete spectrum is the corresponding countable subgroup of the roots
of unity. Finally, we can consider the whole project for the odometer whose
spectrum is the group of all roots of unity of positive integer degrees,
replacing the group~$\mathcal{D}$ by the ad\`ele group and the
graph~$\mathrm{OP}$ by a~more complicated graph, which is yet to be studied.

\subsection{Plan of the paper}
\label{ss1.4}
In~\S\,\ref{s2} we present the necessary definitions and
constructions in the theory of Brattell diagrams, construct the
graph~$\mathrm{OP}$ of ordered pairs, and define the action of the
groups~$\mathbb{Z}$ and $\mathcal{D}=\sum \mathbb{Z}/2\mathbb{Z}$ on the space
$\mathcal{T}(\mathrm{OP})$ of infinite paths of the graph~$\mathrm{OP}$.

In~\S\,\ref{s3} we give an independent description of the path space of the
graph~$\mathrm{OP}$ which is convenient for the study of the action
of~$\mathcal{D}$. We also give a~series of examples of central measures on the
space $\mathcal{T}(\mathrm{OP})$ which, as will be seen below, show the
intermediate growth of the scaling sequences.

In~\S\,\ref{s4} we study the action of an adic transformation
on the space~$\mathcal{T}(\mathrm{OP})$. By making an appropriate change
of variables, we reduce the problem of describing the central measures on the path
space $\mathcal{T}(\mathrm{OP})$ of the graph of ordered pairs~$\mathrm{OP}$
to~the~study of measures on~$I^{\mathbb{Z}}\times I^{\mathbb{N}}$ invariant under the
product of the shift~$\mathbb{S}$ and~the odometer~$\mathbb{O}$.

\S\,\ref{s5} is devoted to the study of the set of measures
on~$I^{\mathbb{Z}}\times I^{\mathbb{N}}$ invariant 
under~$\mathbb{S}\times\nobreak\mathbb{O}$. 
It turns out that such measures are divided into
two types, measures of periodic type and aperiodic measures. We
give a~description of ergodic measures of both types and study their
properties.

In~\S\,\ref{s6} we discuss the scaling sequences of actions of~$\mathbb{Z}$
and~$\mathcal{D}$ on the space $\mathcal{T}(\mathrm{OP})$ with special
measures~$\mu^\sigma$, prove that the scaling sequence of a~filtration is
independent of the choice of the iterated metric, and also describe the
calculation of~a~scaling sequence of the tail filtration on the space
$\mathcal{T}(\mathrm{OP})$ with these special measures.

\section{The graph of ordered pairs and actions on this graph}
\label{s2}

\vskip1pt

\subsection{Bratteli diagrams, central measures and filtration}
\label{ss2.1}
Let $\Gamma$ be a~graded graph (Bratteli diagram) whose floors are finite and
indexed by non-\allowbreak negative integers. The (oriented) edges in this graph can
only connect vertices on adjacent floors, and the beginning of every edge lies
on the floor with the smaller index. The set of vertices on floor~$n$ is denoted
by the symbol~$\Gamma_n$, and the set of edges going from~$\Gamma_n$ to~$\Gamma_{n+1}$ 
is denoted by~$E_n$. We denote by~$\mathcal{T}(\Gamma)$ the set of infinite (oriented)
paths in~$\Gamma $ going from the vertices of the zeroth floor.

The set~$\mathcal{T}(\Gamma)$ is equipped with a~topology in the standard way.
A~base of~this topology is formed by the so-called elementary cylindrical
sets, that is, sets of paths having a~fixed origin. In this topology, the
elementary cylindrical sets are also closed, and the whole
space~$\mathcal{T}(\Gamma)$ is compact. 

For every~$n \geq 0$ we consider the partition $\xi_n$ of~$\mathcal{T}(\Gamma)$ 
into sets of paths that coincide starting at floor~$n$. 
These partitions are in a~sense independent complements of the
algebras of cylindrical sets. Denote by~$\xi$ the \textit{tail filtration}
consisting of~the~sequence of these partitions, $(\xi_n)_{n\geq 0}$. The
equivalence relation generated by~the set-theoretic intersection of all
partitions~$\xi_n$ is called the \textit{tail equivalence relation} of~$\Gamma$: 
two paths are equivalent if and only if they coincide starting 
at some~place. 

An additional (adic) structure on a~graded graph is the reverse lexicographic
ordering of paths (see~\cite{1}). It is defined using introducing
a~total order on the set of edges entering each vertex from the vertices at the
previous level. We say that a~path $y \in \mathcal{T}(\Gamma)$ is greater than
a~path $x \in \mathcal{T}(\Gamma)$ if for some positive integer~$n$ they
coincide starting at floor~$n$ and the edge along which the path~$y$
arrives at floor~$n$ is greater than the edge along which the path~$x$
arrives at the same place. The adic order is a~total order on each class
of tail equivalence. An \textit{adic transformation\/}~$\mathcal{A}$ is defined 
on the partially ordered space of paths $\mathcal{T}(\Gamma)$. 
It assigns to every path the next path in the order. We note that the
next path need not exist for a~given path, and therefore the adic
transformation is defined on a~proper subset of~$\mathcal{T}(\Gamma)$.

One can consider Borel measures on the topological space $\mathcal{T}(\Gamma)$.
We note that the partitions~$\xi_n$ are automatically measurable with respect
to any Borel measure. An important class of measures is formed by the
\textit{central measures\/} (measures with maximum entropy). A~measure~$\mu$
is said to be central if the beginnings of a~path are equiprobable for a~fixed
tail, that is, the conditional measures on the elements of the
partition~$\xi_n$ are uniform. We note that the centrality of a~measure is
equivalent to its invariance under the action of the adic transformation. The
set of all central measures of~$\Gamma$ is denoted by the symbol
$\operatorname{Inv}(\Gamma)$. The set of ergodic central measures is the
Choquet boundary of $\operatorname{Inv}(\Gamma)$ and is called the
\textit{absolute} of~$\Gamma$.

For details concerning the general theory of graded graphs and central
measures, see~\cite{21}--\cite{26}.

\subsection({Description of the graph \$\000\134mathrm\{OP\}\$ of ordered pairs
and of the actions on the path space of the graph}){Description of the
graph~$\mathrm{OP}$ of ordered pairs and of the actions on the path space
of the graph}
\label{ss2.2}

\subsubsection({Description of the graph \$\000\134mathrm\{OP\}\$ of ordered pairs})
{Description of the graph\/~$\mathrm{OP}$ of ordered pairs}
\label{sss2.2.1}
The graded graph~$\mathrm{OP}$
of ordered pairs is defined as follows. The set of vertices
on the zeroth floor is defined as~$\mathrm{OP}_0=I=\{0,1\}$. Further, the set
of vertices on floor $n+1$ is equal
to~$\mathrm{OP}_{n+1}=\mathrm{OP}_n\times \mathrm{OP}_n$, \,$n \geq 0$, that~is,
the set of all possible ordered pairs of vertices on floor~$n$. The set
of edges~$E_n$, $n\geq0$, leading from the vertices on floor~$n$ to those
on floor $n+1$ is constructed as follows. Suppose that $v\in
\mathrm{OP}_{n+1}$, \,$v=(v_0,v_1)$, where~$v_0,v_1\in \mathrm{OP}_{n}$. 
We draw edges from $v_0$ and~$v_1$ to~$v$ and specify the
natural order on these edges, namely, the edge from~$v_0$ to~$v$ is assumed
to be the zeroth and the edge from~$v_1$ to~$v$ the first. If~$v_0=v_1$, then
there are two edges simultaneously from~$v_0$ to the vertex $V=(v_0,v_0)$.

By induction, it can readily be seen that the number of vertices on 
floor~$n$ of~$\mathrm{OP}$ is equal to~$2^{2^n}$, and exactly~$2^n$
paths starting from floor~$0$ enter every vertex on floor~$n$. Since
exactly two edges enter every vertex, except for the vertices on the zeroth 
floor, it follows that the tail filtration~$\xi$ for the graph  
is dyadic.

An adic automorphism~$\mathcal{A}$ (an action of the group~$\mathbb{Z}$) is
defined on the space $\mathcal{T}(\mathrm{OP})$, and one can define 
the canonical (adic) action of the group
$$
\mathcal{D} = \bigoplus_{i=0}^\infty \mathbb{Z}/2\mathbb{Z}.
$$

\subsubsection({Description of the adic transformationon the graph \$\000\134mathrm\{OP\}\$})
{Description of the adic transformation on the graph~$\mathrm{OP}$}
\label{sss2.2.2}
The graph of ordered pairs~$\mathrm{OP}$ is equipped with an
order on the edges entering every vertex, and therefore an adic
transformation~$\mathcal{A}$ is defined automatically (see the general
definition in~\S\,\ref{ss2.1}). More specifically, for a~path
$x\in\mathcal{T}(\mathrm{OP})$ passing through vertices $v_n \in
\mathrm{OP}_n$, $n\geq 0$, we can find the first edge of~$x$ whose index
is~$0$. Let this edge pass from a~vertex~$v_n$ to~$v_{n+1}$. We define the path
$\mathcal{A}x$ in such a~way that it coincides with~$x$ starting from floor $n+1$, 
enters~$v_{n+1}$  
along an edge with index~$1$, and the previous edges all have index~$0$.   
This transformation is defined only for
paths having at least one edge with index~$0$.

\subsubsection
({Description of the canonical action of the group \$\000\134mathcal\{D\}\$
on the graph \$\000\134mathrm\{OP\}\$})
{Description of the canonical action of the group~$\mathcal{D}$
on the graph\/~$\mathrm{OP}$}
\label{sss2.2.3}
The simplest way to define the
canonical (adic) action~$\kappa$ of~$\mathcal{D}$ is to define it on the
generators. We denote by~$\kappa(g)$ the action of an element~$g$
of~$\mathcal{D}$ on~$\mathcal{T}(\mathrm{OP})$. Let the group $\mathcal{D} =
\bigoplus_{i=0}^\infty \mathbb{Z}/2\mathbb{Z}$ be generated by the elements~$g_i$,
$i \geq 0$. We denote by~the symbol~$\mathcal{D}_n$, $n \geq 0$, the finite
subgroup generated by~$g_0, \dots, g_{n-1}$. Let $g_n$ be a~generator
of~$\mathcal{D}$ and let $x\in \mathcal{T}(\mathrm{OP})$ be a~path. We
define a~path $\kappa(g_n)x$ in such a~way that it coincides with~$x$, starting  
at floor~$n+1$, the index of its edge entering the vertex
on floor $n+1$ differs from the index of the corresponding edge~in the~path~$x$, and
the indices of the edges between floors~$i$ and~$i+1$, \,$0\leq i<n$, 
in~the~paths~$x$ and~$\kappa(g_n)x$ coincide. The actions (of the generators) defined
in this way satisfy the commutation relations of the group~$\mathcal{D}$, and
therefore define an~action.

\vskip1pt

We note that the action~$\kappa$ of~$\mathcal{D}$ is closely related to the
tail filtration~$\xi$ and to the set of central measures
$\operatorname{Inv}(\mathrm{OP})$. Namely, the partition~$\xi_n$ is precisely
the partition into orbits of the action of the subgroup~$\mathcal{D}_n$. The
partition into orbits of the whole group~$\mathcal{D}$ is the tail partition
for the graph~$\mathrm{OP}$. A~Borel measure on the space
$\mathcal{T}(\mathrm{OP})$ is central if and only if it is invariant 
under the action~$\kappa$ of~$\mathcal{D}$.

\section({Independent description of the path space of the
graph~\$\000\134mathrm\{OP\}\$}){Independent description of the path space
of the graph~$\mathrm{OP}$}
\label{s3}

\vskip1pt

\subsection({The space \$I\000\136\{\000\134mathcal\{D\}\}\000\134times
I\000\136\{\000\134mathbb\{N\}\}\$, the action of the
group~\$\000\134mathcal\{D\}\$, the filtration and the isomorphism
theorem}){The space $I^{\mathcal{D}}\times I^{\mathbb{N}}$, the action of the
group~$\mathcal{D}$, the filtration and the isomorphism theorem}
\label{ss3.1}
Consider the space $I^{\mathcal{D}}\,{\times}\, I^{\mathbb{N}}$,
consisting of all possible pairs~$(w,\alpha)$, where $w\in I^{\mathcal{D}}$ is
a~configuration on the group~$\mathcal{D}$ and~$\alpha\in I^{\mathbb{N}}$ is an infinite
sequence of zeros and ones. The cylindrical sets on this space, and also the
topology of the product of compact spaces, are defined in the standard way. The
cylindrical sets are open and closed in this topology.

The space $I^{\mathbb{N}}$ with coordinatewise addition forms a~commutative
group, and the group~$\mathcal{D}$ is isomorphic to the subgroup
of~$I^{\mathbb{N}}$ consisting of finite sequences (with zeros starting 
at some place).

\begin{definition}
\label{d1}
We denote by~$\tau$ the embedding of the group~$\mathcal{D}$ in the
group~$I^{\mathbb{N}}$ which takes every element $g \in \mathcal{D}$, \,$g= \sum
\alpha_i g_{i-1}$, to the finite sequence $(\alpha_i)_{i \geq 1}$
of the coefficients of its expansion in the generators.
\end{definition}

The group~$\mathcal{D}$ acts on the space $I^{\mathcal{D}}$ by 
shifting the argument, 
$$
w(\,\cdot\,) \mapsto w(\,\cdot\,+g), \qquad g \in \mathcal{D}, \quad w
\in I^{\mathcal{D}},
$$
and on the space $I^{\mathbb{N}}$ by addition, 
$$
\alpha\mapsto \alpha+\tau(g), \qquad \alpha \in I^{\mathbb{N}},
\quad g \in \mathcal{D}.
$$

The direct product of these actions, which is the action of~$\mathcal{D}$
on the product $I^{\mathcal{D}}\times I^{\mathbb{N}}$, is given by the formula
\begin{equation}
\label{eq1}
\operatorname{diag}(g)\colon \bigl(w(\,\cdot\,),\alpha\bigr) \mapsto
\bigl(w(\,\cdot\,+g), \alpha+\tau(g)\bigr), \qquad g \in \mathcal{D}, \quad
w \in I^{\mathcal{D}}, \quad\alpha \in I^{\mathbb{N}}.
\end{equation}

This action and the dyadic structure (the sequence of embedded subgroups
$\mathcal{D}_n$) of the group~$\mathcal{D}$ define a~dyadic filtration on the
space $I^{\mathcal{D}}\times I^{\mathbb{N}}$, which is the tail filtration for
this action. Namely, for $n \geq 0$ we define the partition $\zeta_n$ of the
space $I^{\mathcal{D}}\times\nobreak I^{\mathbb{N}}$ as the partition into orbits
of the action of the subgroup~$\mathcal{D}_n$. It should be noted that this
filtration is not the direct product of the tail filtrations of the corresponding
actions of~$\mathcal{D}$ on~$I^{\mathcal{D}}$ and~$I^{\mathbb{N}}$.

\begin{theorem}
\label{t1}
The path space\/ $\mathcal{T}(\mathrm{OP})$ of the graph\/~$\mathrm{OP}$
of ordered pairs with the action\/~$\kappa$ and the space\/ $I^{\mathcal{D}}\times
I^{\mathbb{N}}$ with the action\/ $\operatorname{diag}$ are isomorphic
as topological\/ $\mathcal{D}$-spaces.
\end{theorem}

To prove this theorem, we need some preliminary considerations.

\subsection{Labels on the vertices of the graph and 
proof of Theorem~\ref{t1}}
\label{ss3.2}
The construction of the graph~$\mathrm{OP}$ of ordered pairs
enables us to parametrize its vertices 
by configurations on finite subgroups of~$\mathcal{D}$ in a~natural way. 
We inductively define
a~map~$\Phi$ which assigns configurations on~$\mathcal{D}_n$, $n\geq 0$, 
bijectively to vertices on floor~$n$. When $v \in \mathrm{OP}_0
= I$ we define $\Phi[v]$ as the configuration that takes the single element
of the subgroup $\mathcal{D}_0=\{0\}$ to~$v$. Further, suppose that we have
already defined~$\Phi$ on~$\mathrm{OP}_n$. Let $v \in
\mathrm{OP}_{n+1}$, \,$v = (v_0,v_1)$, where $v_0,v_1 \in\nobreak \mathrm{OP}_n$. 
We define the configuration~$\Phi[v]$ on~$\mathcal{D}_{n+1}$ by the relation
\begin{equation}
\label{eq2}
\Phi[v](h) =
\begin{cases}
\Phi[v_0](h),& h \in \mathcal{D}_n,
\\[2pt]
\Phi[v_1](g_n + h),& h \in \mathcal{D}_{n+1}\setminus \mathcal{D}_n.
\end{cases}
\end{equation}

\begin{remark}
\label{r1} For each $n \geq 0$ the map~$\Phi$ is a~bijection between the
floor~~$\mathrm{OP}_n$ of the graph~$\mathrm{OP}$ and the set
$I^{\mathcal{D}_n}$ of configurations on the subgroup~$\mathcal{D}_n$.
\end{remark}

We note that for two vertices $v\in \mathrm{OP}_n$ and $v' \in
\mathrm{OP}_{n+1}$ joined by an edge with index $\beta \in \{0,1\}$, the
corresponding configurations on~$\mathcal{D}_n$ are connected by the relation
\begin{equation}
\label{eq3}
\Phi[v](h) =\Phi[v'](h + \beta g_n), \qquad h \in \mathcal{D}_n.
\end{equation}

\begin{proof}[of Theorem~\ref{t1}]
To prove the theorem, we shall explicitly construct an isomorphism~$\Psi \colon
\mathcal{T}(\mathrm{OP}) \to I^{\mathcal{D}}\times I^{\mathbb{N}}$.
To do this, we define two maps, a~map~$F$ of~$\mathcal{T}(\mathrm{OP})$ 
onto $I^{\mathcal{D}}$ and a~map~$A$ of~$\mathcal{T}(\mathrm{OP})$ 
onto~$I^{\mathbb{N}}$. 
Then the map $\Psi = (F, A)$ will be the desired homeomorphism 
of~$\mathcal{T}(\mathrm{OP})$ onto $I^{\mathcal{D}}\times I^{\mathbb{N}}$.

\vskip1pt

Let $x \in \mathcal{T}(\mathrm{OP})$ be the path passing through the vertices
$v_n \in \mathrm{OP}_n$, $n \geq 0$, with edges $e_n \in E_n$, \,$n \geq 0$
\,($e_n$ joins $v_n$ and~$v_{n+1}$). We recall that, by the
construction of the graph of ordered pairs, an order is given on the two edges
leading to a~given vertex (the numbers 0 and~1). For $n \geq 1$ we
set~$\alpha_n$ equal to the ordinal number of the edge $e_{n-1}$. This
defines a~sequence $\alpha = (\alpha_n)_{n \geq 1} \in I^{\mathbb{N}}$, which
we call the image of~$x$ under the map~$A$:
$$
A[x]:=\alpha.
$$

The map~~$F$ is defined in a~somewhat more complicated way. We have to define
a~configuration $F[x] \in I^{\mathcal{D}}$ on the group~$\mathcal{D}$. Here, 
every vertex~$v_n$, $n \geq 0$, of the path~$x$ is taken by the map~$\Phi$
to the configuration~$\Phi[v_n]$ on the finite subgroup~$\mathcal{D}_n$. We note that
the sequence of configurations~$\Phi[v_n]$, $n \geq 0$, is not consistent
in general, that is, the configuration~$\Phi[v_n]$ is not the restriction of the
configuration~$\Phi[v_{n+1}]$ to the subgroup~$\mathcal{D}_n$. A~direct
consequence of the equation~\eqref{eq3} and the definition of the
number~$\alpha_{n+1}$ is the identity
$$
\Phi[v_n](\,\cdot\,) = \Phi[v_{n+1}](\,\cdot\,+\alpha_{n+1} g_n) \quad
\text{on}\quad \mathcal{D}_n, \quad n\geq 0.
$$

This implies the relation
$$
\Phi[v_n]\biggl(\,\cdot\,+\sum_{i=0}^{n-1}\alpha_{i+1} g_i\biggr) =
\Phi[v_{n+1}]\biggl(\,\cdot\,+\sum_{i=0}^{n}\alpha_{i+1} g_i\biggr) \quad
\text{on}\quad \mathcal{D}_n, \quad n\geq 0,
$$
which means that the sequence of configurations
$\Phi[v_n]\bigl(\,\cdot\,+\sum_{i=0}^{n-1}\alpha_{i+1} g_i\bigl)$, $n \geq 0$,
is consistent. Therefore, we can specify a~configuration on~~$\mathcal{D}$
whose restriction to every subgroup~$\mathcal{D}_n$ coincides with the
configuration $\Phi[v_n]\bigl(\,\cdot\,+\sum_{i=0}^{n-1}\alpha_{i+1}
g_i\bigr)$. We define this configuration on~$\mathcal{D}$ to be the
image~$F[x]$ of the path~$x$ under the~map~$F$:
\begin{equation}
\label{eq4}
F[x](g) = \Phi[v_n]\biggl(g+\sum_{i=0}^{n-1}\alpha_{i+1} g_i\biggr), \qquad
g \in \mathcal{D}_n, \quad n \geq 1.
\end{equation}

As already noted above, we define the map~$\Psi$ as a~pair~$(F,A)$:
$$
\Psi[x]=(F[x],A[x]), \qquad x \in \mathcal{T}(\mathrm{OP}).
$$
We claim that $\Psi$ is a~bijection of the path space
$\mathcal{T}(\mathrm{OP})$ of the graph~$\mathrm{OP}$ onto $I^{\mathcal{D}}
\times I^{\mathbb{N}}$. We note that formula~\eqref{eq4} enables us
to recover uniquely the vertices~$v_n$, $n \geq 0$, of~$x$ from the
configuration~$F[x]$ and the sequence $\alpha = A[x]$. Further, the edge~$e_n$,
$n \geq 0$, can be recovered uniquely from its end~$v_{n+1}$ and the ordinal number
~$\alpha_{n+1}$. Moreover, the configuration can be chosen arbitrarily
in~$I^{\mathcal{D}}$, and the sequence~$\alpha$ can be chosen arbitrarily
in~$I^{\mathbb{N}}$. This proves that the map~$\Psi$ is a~bijection.

It is obvious from the definition of~$\Psi$ that the image of a~cylindrical set 
in~$\mathcal{T}(\mathrm{OP})$ under~$\Psi$ 
is a~cylindrical set in~$I^{\mathcal{D}}\times\nobreak I^{\mathbb{N}}$.
This implies that $\Psi$
is a~homeomorphism.

It remains to prove that $\Psi$ is~$\mathcal{D}$-equivariant. Let us
verify the necessary commutation relation for an arbitrary $m \geq 0$:
\begin{equation}
\label{eq5}
\Psi[\kappa(g_m)x]=\operatorname{diag}(g_m)\Psi[x], \qquad x \in
\mathcal{T}(\mathrm{OP}).
\end{equation}

By the definition of the canonical action~$\kappa$, the path~$\kappa(g_m)x$
coincides with the path~$x$ starting at floor~$m+1$, the ordinal
numbers of the edges in these paths coincide up~to floor~$m$, and the
ordinal number of the edge going from floor~$m$ to floor~$m+1$ is
different. By the definition of~$A$, this implies that
$$
A[\kappa(g_m)x]_i = A[x]_i, \quad i \ne
m+1,\qquad A[\kappa(g_m)x]_{m+1} \ne A[x]_{m+1}.
$$

In other words, $A[\kappa(g_m)x] = A[x]+\tau(g_m)$. The equation
$F[\kappa(g_m)x](\,\cdot\,)=F[x](\,\cdot\,+g_m)$ follows from 
formula~\eqref{eq4}. Thus, we have proved relation~\eqref{eq5} and the
$\mathcal{D}$-equivariance of~$\Psi$. This completes the proof
of Theorem~\ref{t1}. 
\end{proof}

\subsection({Examples of central measures on the path
space \$\000\134mathcal\{T\}(\000\134mathrm\{OP\})\$ of the
graph \$\000\134mathrm\{OP\}\$, and the
measures \$\000\134mu\000\136\000\134sigma\$}){Examples of central measures
on the path space~$\mathcal{T}(\mathrm{OP})$ of the graph~$\mathrm{OP}$, 
and the measures~$\mu^\sigma$}
\label{ss3.3}
One of the central problems of the theory
of graded graphs is that of describing the central measures on the path
space of a~graph. \S\,\ref{s5} below is devoted to a~detailed study of the set
$\operatorname{Inv}(\mathrm{OP})$ of central measures on the path space
$\mathcal{T}(\mathrm{OP})$ of~$\mathrm{OP}$. In this subsection 
we discuss a~class of examples of central measures.

The equivariant homeomorphism~$\Psi$ takes the set
$\operatorname{Inv}(\mathrm{OP})$ of central measures to the set of measures
on~$I^{\mathcal{D}}\times I^{\mathbb{N}}$ invariant under the action
$\operatorname{diag}$ of the group~$\mathcal{D}$ given by 
formula~\eqref{eq1}. The projections of these measures to~$I^{\mathcal{D}}$
and~$I^{\mathbb{N}}$ are also invariant under the corresponding actions
of~$\mathcal{D}$. The only measure on~$I^{\mathbb{N}}$ invariant under this
action of~~$\mathcal{D}$ is the Lebesgue measure~$m$.

Measures of the form $\mu\times m$, where $\mu$ is a~$\mathcal{D}$-invariant
measure on~$I^{\mathcal{D}}$ and $m$ is the Lebesgue measure
on~$I^{\mathbb{N}}$, can serve as examples of measures
on~$I^{\mathcal{D}}\times I^{\mathbb{N}}$ that are invariant under the action
$\operatorname{diag}$ of~$\mathcal{D}$. Among these measures, we
single out a~special class.

\begin{definition}
\label{d2} 
Let $\sigma = (\sigma_i)_{i\geq 0}$ be a~sequence of zeros and
ones. Let
$$
\mathcal{D}^\sigma = \langle g_i \colon \sigma_i=0, \, i \geq 0 \rangle
$$
be the subgroup of~$\mathcal{D}$ generated by those elements~$g_i$ for
which $\sigma_i=\nobreak0$. Consider the quotient map~$\mathcal{D} \to
\mathcal{D}/\mathcal{D}^\sigma$ and the map of configurations
$I^{\mathcal{D}/\mathcal{D}^\sigma}\to I^{\mathcal{D}}$ induced by~it. 
We denote by~$m^\sigma$ the measure on~$I^{\mathcal{D}}$ which is the
pushforward of the Lebesgue measure on~$I^{\mathcal{D}/\mathcal{D}^\sigma}$
under the map of configurations, and by~$\omega^\sigma$ the product
$m^\sigma \times m$ of~$m^\sigma$ and the Lebesgue measure on~$I^{\mathbb{N}}$.
\end{definition}

The measures $\mu^\sigma$ on the path space $\mathcal{T}(\mathrm{OP})$ 
of~$\mathrm{OP}$ (the pushforward of the measures~$\omega^\sigma$ under
the map~$\Psi^{-1}$) were studied in~\cite{27}. The scaling sequences (see
\S\,\ref{s6}) of the adic automorphism on the space
$(\mathcal{T}(\mathrm{OP}),\mu^\sigma)$ were calculated there, as well as the
scaling sequence of the canonical action~$\kappa$ of~$\mathcal{D}$ on this
space.

\section({Action of the group~\$\000\134mathbb\{Z\}\$}){Action
of the group~$\mathbb{Z}$}
\label{s4}

\vskip1pt

\subsection({Action of the group~\$\000\134mathbb\{Z\}\$
on~\$I\000\136\{\000\134mathcal\{D\}\}\000\134times
I\000\136\{\000\134mathbb\{N\}\}\$}){Action of the group~$\mathbb{Z}$
on~$I^{\mathcal{D}}\times I^{\mathbb{N}}$}
\label{ss4.1}
We proceed with the
study of the action of~$\mathbb{Z}$ on the path space
$\mathcal{T}(\mathrm{OP})$ of the graph~$\mathrm{OP}$, that is, the adic
transformation~$\mathcal{A}$ defined in~\S\,\ref{ss2.2}.

\begin{definition}
\label{d3} 
Let~$I^{\mathbb{N}}_0$ denote the set of all sequences
in~$I^{\mathbb{N}}$ containing only finitely many zeros or ones and
$I^{\mathbb{N}}_\infty$ the complement of~$I^{\mathbb{N}}_0$, that is, the
set of sequences containing infinitely many zeros and ones.
\end{definition}

Recall the definition of odometer.

\begin{definition}
\label{d4} 
The \textit{odometer} is a~transformation~$\mathbb{O}$ defined
on the set~$I^{\mathbb{N}}$. Let $\alpha \in\nobreak I^{\mathbb{N}}$ be such that
$\alpha_1=\alpha_2=\dots=\alpha_{n-1}=\nobreak1$ and $\alpha_{n}=\nobreak0$ 
for some $n\geq 0$.
Then the sequence $\mathbb{O}[\alpha]$ is defined in such a~way that
$\mathbb{O}[\alpha]_i=\alpha_i$ when $i>n$, $\mathbb{O}[\alpha]_{n}=\nobreak1$, 
and $\mathbb{O}[\alpha]_i=\nobreak0$ when $i< n$.
\end{definition}

The domain of the odometer in this definition is, generally speaking, not the
whole of~$I^{\mathbb{N}}$. We note that $I^{\mathbb{N}}$ is the
group~$\mathbb{Z}_2$ of integer $2$-adic numbers in the~`digital'
representation. Here the odometer~$\mathbb{O}$ is the addition of one in the
group~$\mathbb{Z}_2$. The odometer defined on~$I^{\mathbb{N}}$ using the
structure of~$2$-adic numbers is already defined on the whole set. However, we
will not use this extension but confine ourselves to Definition~\ref{d4}.
This approach is more convenient when working with the graph~$\mathrm{OP}$.

\begin{remark}
\label{r2} 
All the $\mathbb{O}^n[\alpha]$, $n \in \mathbb{Z}$, are defined for an
$\alpha\in I^{\mathbb{N}}$ if and only if~$\alpha \in I^{\mathbb{N}}_\infty$.

\end{remark}

We recall that the adic transformation is also defined only on a~subset of the
path space $\mathcal{T}(\mathrm{OP})$ of the graph~$\mathrm{OP}$ rather than
on the whole space.

\begin{definition}
\label{d5} 
Let $\mathcal{T}_0(\mathrm{OP})\subset \mathcal{T}(\mathrm{OP})$ be
the set of all possible paths $x\in \mathcal{T}(\mathrm{OP})$ for which all
the $\mathcal{A}^nx$, $n \in \mathbb{Z}$, are defined.
\end{definition}

\begin{remark}
\label{r3} 
Let $x\in\mathcal{T}(\mathrm{OP})$. Then
$x\in\mathcal{T}_0(\mathrm{OP})$ if and only if~$A[x] \in
I^{\mathbb{N}}_\infty$. The map~$\Psi$ establishes a~bijection between
$\mathcal{T}_0(\mathrm{OP})$ and $I^{\mathcal{D}}\times I^{\mathbb{N}}_\infty$.
\end{remark}

We now study the transformation $\mathcal{A}^\Psi= \Psi \circ \mathcal{A}\circ
\Psi^{-1}$ on~$I^{\mathcal{D}}\times I^{\mathbb{N}}_\infty$. It follows from
the definition of the adic transformation that for $x \in
\mathcal{T}_0(\mathrm{OP})$ the sequence $A[\mathcal{A}x]$ is the sequent
to~$A[x]$ in the reverse lexicographic order. In other words, the
transformation $\mathcal{A}^\Psi\colon I^{\mathcal{D}}\times
I^{\mathbb{N}}_\infty \to I^{\mathcal{D}}\times I^{\mathbb{N}}_\infty$ acts
on the second coordinate as the odometer~$\mathbb{O}$. On the other hand, the
orbits of the actions of the groups $\mathbb{Z}$ and~$\mathcal{D}$ on the space
$\mathcal{T}_0(\mathrm{OP})$ coincide, and both actions are free. Hence,
for every $x \in\nobreak \mathcal{T}_0(\mathrm{OP})$ there is a~unique element 
$g \in \mathcal{D}$ such that $\mathcal{A}x = \kappa(g) x$. Then
$\mathbb{O}[A[x]]=A[\mathcal{A}x] = A[\kappa(g) x]=A[x]+\tau(g)$, that is, $g =
\tau^{-1}(\mathbb{O}[A[x]]-A[x])$. Consequently,
$$
F[\mathcal{A}x](\,\cdot\,) = F[\kappa(g) x](\,\cdot\,) = F[x](\,\cdot\,+g).
$$

Thus, we have proved the following assertion.

\begin{proposition}
\label{p1} 
The transformation\/ $\mathcal{A}^\Psi$ on\/~$I^{\mathcal{D}}\times
I^{\mathbb{N}}_\infty$ satisfies the relation
\begin{equation}
\label{eq6}
\mathcal{A}^\Psi (f,\alpha)= \bigl(f\bigl(\,\cdot\,+ \tau^{-1}(\mathbb{O}[\alpha]-\alpha)\bigr),
\mathbb{O}[\alpha]\bigr), \qquad f \in I^{\mathcal{D}},\quad \alpha
\in I^{\mathbb{N}}_\infty.
\end{equation}
\end{proposition}

\subsection({Isomorphism of the
transformation~\$\000\134mathcal\{A\}\000\136\000\134Psi\$
and the direct product of the actions
on~\$I\000\136\{\000\134mathbb\{Z\}\}\000\134times
I\000\136\{\000\134mathbb\{N\}\}\$. The map~\$\000\134Lambda\$}){Isomorphism
of the transformation $\mathcal{A}^\Psi$ and the direct product of the actions
on~$I^{\mathbb{Z}}\times I^{\mathbb{N}}$. The map~$\Lambda$}
\label{ss4.2}
For a~more detailed study of the action of the
group~$\mathbb{Z}$ it is reasonable
to replace the phase space $I^{\mathcal{D}}\times I^{\mathbb{N}}$
by~$I^{\mathbb{Z}}\times I^{\mathbb{N}}$. Here we shall do this in such a~way that
the transformation $\mathcal{A}^\Psi$ will pass to the transformation
$\mathbb{S}\times\mathbb{O}$, where $\mathbb{S}$~stands for the left shift
on~$I^{\mathbb{Z}}$.

\vskip1pt

We want to construct a~map~$\Lambda\colon I^{\mathcal{D}}\times
I^{\mathbb{N}}_\infty \to I^{\mathbb{Z}}\times I^{\mathbb{N}}_\infty$ which
closes the commutative diagram
\begin{equation}
\begin{gathered}
\xymatrix{
\mathcal{T}_0(\mathrm{OP}) \ar[r]^{\Psi}
\ar[d]^{\mathcal{A}} &I^{\mathcal{D}}\times I^{\mathbb{N}}_\infty
\ar[r]^{\Lambda} \ar[d]^{\mathcal{A}^\Psi} &I^{\mathbb{Z}}\times I^{\mathbb{N}}_\infty
\ar[d]^{\mathbb{S}\times\mathbb{O}}
\\
\mathcal{T}_0(\mathrm{OP}) \ar[r]^{\Psi} &I^{\mathcal{D}}\times I^{\mathbb{N}}_\infty
\ar[r]^{\Lambda} &I^{\mathbb{Z}}\times I^{\mathbb{N}}_\infty
}
\end{gathered}
\label{eq7}
\end{equation}

If a~measure~$\nu$ on~$I^{\mathcal{D}}\times I^{\mathbb{N}}$ is invariant 
under the action of~$\mathcal{D}$, then the set $I^{\mathcal{D}}\times
I^{\mathbb{N}}_\infty$ has the full $\nu$-measure since the projection
of~$\nu$ to~$I^{\mathbb{N}}$ is the Lebesgue measure~$m$. Thus, the
map~$\Lambda$ is defined $\operatorname{mod}0$ with respect to~$\nu$.

\begin{definition}
\label{d6} 
A~Borel measurable map from the product $I^{\mathcal{D}}\times
I^{\mathbb{N}}_\infty$ to the product $I^{\mathbb{Z}}\times
I^{\mathbb{N}}_\infty$ is said to be \textit{fibrewise} if it preserves the
second component, that is, for every $\alpha \in I^{\mathbb{N}}_\infty$ the
image of the fibre $I^{\mathcal{D}}\times\{\alpha\}$ is contained in the fibre
$I^{\mathbb{Z}} \times\{\alpha\}$.

\end{definition}

We restrict ourselves to the class of fibrewise maps~$\Lambda$ whose fibres
$\Lambda(\,\cdot\,,\alpha)\colon 
I^{\mathcal{D}}\to I^{\mathbb{Z}}$ are induced
for $\alpha \in I^{\mathbb{N}}_\infty$ by some maps $\lambda_\alpha\colon
\mathbb{Z} \to \mathcal{D}$, that is, are changes of the argument of the
configurations:
\begin{equation}
\label{eq8}
\Lambda(f,\alpha) = \bigl(f\circ \lambda_\alpha,\alpha), \qquad f
\in I^{\mathcal{D}}, \quad\alpha \in I^{\mathbb{N}}_\infty.
\end{equation}

For such maps~$\Lambda$ commutative diagram~\eqref{eq7} is obviously
equivalent to the following relation on the family of maps
$(\lambda_\alpha)_{\alpha \in I^{\mathbb{N}}_\infty}$:
\begin{equation}
\label{eq9}
\lambda_{\mathbb{O}[\alpha]}(k)+\tau^{-1}(\mathbb{O}[\alpha]-\alpha)=\lambda_\alpha(k+1),
\qquad \alpha \in I^{\mathbb{N}}_\infty, \quad k \in \mathbb{Z}.
\end{equation}

\begin{lemma}
\label{l1} 
The map\/~$\Lambda \colon I^{\mathcal{D}}\times
I^{\mathbb{N}}_\infty\,{\to}\,I^{\mathbb{Z}}\times I^{\mathbb{N}}_\infty$
defined by equation~\eqref{eq8} and satisfying commutation
relation~\eqref{eq7} is Borel measurable if and only if the map\/
$\lambda\colon\alpha \mapsto \lambda_\alpha(0)$ from\/~$I^{\mathbb{N}}_\infty$
to\/~$\mathcal{D}$ is Borel measurable. Here the map\/~$\Lambda$ is completely
determined by the map\/~$\lambda$.
\end{lemma}

\begin{proof}
Let $k \in \mathbb{Z}$, \,$w^0\in I$. \,Let $W=\{w \in I^{\mathbb{Z}}\colon
w(k)=w^0\}$. \,Then
\begin{align}
\Lambda^{-1}(W\times I^{\mathbb{N}}_\infty) = &\bigcup_{\alpha
\in I^{\mathbb{N}}_\infty} \{f \in I^{\mathcal{D}}\colon
f(\lambda_\alpha(k))=w^0\}\times
\{\alpha\}
\nonumber
\\
=&\bigcup_{g\in \mathcal{D}} \{f \in I^{\mathcal{D}}\colon f(g)=w^0\} \times
\{\alpha \in I^{\mathbb{N}}_\infty\colon \lambda_\alpha(k)=g\}.
\label{eq10}
\end{align}
The measurability of~$\Lambda$ is equivalent to the condition that the
set in formula~\eqref{eq10}~is measurable for every $k \in \mathbb{Z}$,
\,$w^0\in I$.

If~$\Lambda$ is measurable, then the set in formula~\eqref{eq10} is
measurable when $k=\nobreak0$ and $w^0=\nobreak1$. Let $g_0 \in \mathcal{D}$ and choose
$f^0\in I^{\mathcal{D}}$ in such a~way that $f^0(g)=\nobreak1$ holds for $g \in
\mathcal{D}$ only when $g=g_0$. Then the set
\begin{align}
&\{\alpha \in I^{\mathbb{N}}_\infty\colon (f^0,\alpha) \in
\Lambda^{-1}(W\times I^{\mathbb{N}}_\infty)\} = \{\alpha
\in I^{\mathbb{N}}_\infty\colon
f^0\circ\lambda_\alpha \in W\}
\nonumber
\\[1pt]
&\qquad=\{\alpha \in I^{\mathbb{N}}_\infty\colon f^0(\lambda_\alpha(0))=1\}=
\{\alpha \in I^{\mathbb{N}}_\infty\colon \lambda_\alpha(0)=g_0\}
\nonumber
\\[1pt]
&\qquad=\{\alpha \in I^{\mathbb{N}}_\infty\colon \lambda(\alpha)=g_0\}
\label{eq11}
\end{align}
is measurable. Thus, $\lambda$ is measurable.

If~$\lambda$ is measurable, then it follows from formula~\eqref{eq9} 
that for every $k \in\nobreak \mathbb{Z}$ the map~$\alpha\mapsto
\lambda_{\alpha}(k)$ from $I^{\mathbb{N}}_\infty$ to~$\mathcal{D}$ is also
measurable. However, in this case, the set on the right-hand side of
formula~\eqref{eq10} is a~countable union of measurable sets, that is, it is
measurable. Hence $\Lambda$ is measurable.

The last assertion of the lemma is trivial, since the maps~$\lambda_\alpha$ are
uniquely defined by~$\lambda$ via formula~\eqref{eq9}. 
\end{proof}

Choose a~map~$\lambda \colon I^{\mathbb{N}}_\infty \to \mathcal{D}$
which takes everything to the zero element of the group~$\mathcal{D}$. By
relations~\eqref{eq9}, this map generates a~family of maps $\lambda_\alpha
\colon \mathbb{Z}\to \mathcal{D}$, $\alpha \in\nobreak I^{\mathbb{N}}_\infty$, 
which, by formula~\eqref{eq8}, in turn defines a~fibrewise map~$\Lambda \colon
I^{\mathcal{D}}\times I^{\mathbb{N}}_\infty \to I^{\mathbb{Z}}\times\nobreak
I^{\mathbb{N}}_\infty$. We fix the symbols $\lambda_\alpha$ and~$\Lambda$ for
these objects.

Thus, the action of the adic transformation~$\mathcal{A}$ on the path space
$\mathcal{T}(\mathrm{OP})$ of the~graph~$\mathrm{OP}$ is Borel isomorphic
to the action of the transformation $\mathbb{S}\times\mathbb{O}$ on~the space
$I^{\mathbb{Z}}\times I^{\mathbb{N}}$.

\vskip1pt

This implies the following theorem on the universal adic approximation 
of automorphisms.

\begin{theorem}
\label{t2} 
Let\/~$T$ be an automorphism of a~Lebesgue space\/ $(X,p)$ with
a~binomial generator. Then there is a~measure\/ $\mu \in
\operatorname{Inv}(\mathrm{OP})$ such that the adic
transformation\/~$\mathcal{A}$ on the path space\/ $\mathcal{T}(\mathrm{OP})$
of the graph\/~$\mathrm{OP}$ with the measure\/~$\mu$ is isomorphic to the
transformation\/ $T\times \mathbb{O}$ on the direct product\/ $X \times
I^{\mathbb{N}}$ with the measure\/ $p\times\nobreak m$.
\end{theorem}

\begin{definition}
\label{d7} 
We denote by~$\mathbb{M}$ the set of Borel probability measures
on the space $I^{\mathbb{Z}}\times I^{\mathbb{N}}$ that are invariant 
under the transformation $\mathbb{S}\times\mathbb{O}$ and 
by~$\operatorname{Ex}(\mathbb{M})$ the set of ergodic measures in~$\mathbb{M}$,
which is the Choquet boundary consisting of the extreme points.
\end{definition}

\begin{remark}
\label{r4} 
The map~$\Lambda \circ \Psi$ is a~bijection between the set
of central measures $\operatorname{Inv}(\mathrm{OP})$ on the path space
$\mathcal{T}(\mathrm{OP})$ of the graph~$\mathrm{OP}$ and the set~$\mathbb{M}$.
\end{remark}

The following remark gives an explicit formula for the fibrewise map~$\Lambda$.

\begin{remark}
\label{r5} 
For every $\alpha \in I^{\mathbb{N}}_\infty$ the
map~$\lambda_\alpha$ satisfies the relation
\begin{equation}
\label{eq12}
\lambda_\alpha \biggl(\sum_{i=0}^n
(-1)^{\alpha_{i+1}} \beta_i 2^i\biggr) = \sum_{i=0}^n
\beta_i g_i,
\end{equation}
where $n\geq 0$ and $\beta_i \in\{0,1\}$.

The pre-image of the subgroup $\mathcal{D}_n$ under the map~$\lambda_\alpha$ is
the segment of integers $\{-a_n,\dots, -a_n+2^n-1\}$, where
$a_n=\sum_{i=1}^{n}\alpha_i 2^{i-1}$. Here $\lambda_\alpha^{-1}(\mathcal{D}_n)$
is the left half of~$\lambda_\alpha^{-1}(\mathcal{D}_{n+1})$
if~$\alpha_{n+1}=\nobreak0$ and the right half if~$\alpha_{n+1}=\nobreak1$.
\end{remark}

\section({Invariant measures
on \$I\000\136\{\000\134mathbb\{Z\}\}\000\134times
I\000\136\{\000\134mathbb\{N\}\}\$}){Invariant measures
on~$I^{\mathbb{Z}}\times I^{\mathbb{N}}$}
\label{s5}

\vskip1pt

\subsection{Direct products and measures of periodic type}
\label{ss5.1}
In this subsection, we study the measures on~$I^{\mathbb{Z}}\times
I^{\mathbb{N}}$ that are invariant under the action
of~$\mathbb{S}\times\mathbb{O}$. We recall that we have denoted the set of these
measures by~$\mathbb{M}$. This set of measures is taken under the
map~$\Lambda^{-1}$ to the set of all measures on~$I^{\mathcal{D}}\times
I^{\mathbb{N}}$ invariant under the action $\operatorname{diag}$ of the
group~$\mathcal{D}$. Under the map~$\Psi^{-1}\circ \Lambda^{-1}$, this set
passes into the set of central measures on the path space
$\mathcal{T}(\mathrm{OP})$ of the graph~$\mathrm{OP}$ of ordered pairs.

The simplest examples of measures in~$\mathbb{M}$ are the direct products.
If~$\eta$ is a~Borel probability measure on~$I^{\mathbb{Z}}$ invariant under the
shift~$\mathbb{S}$, then its direct product with the Lebesgue measure~$m$
on~$I^{\mathbb{N}}$, that is, the measure $\eta\times m$, is invariant 
under the transformation $\mathbb{S}\times\mathbb{O}$.

\begin{definition}
\label{d8} 
Let $\nu$ be a~measure on~$I^{\mathbb{Z}}\times I^{\mathbb{N}}$.
We denote by~$\mathcal{P}[\nu]$ the projection of~$\nu$
to~$I^{\mathbb{Z}}$. Let $k \geq 0$ and $0\leq r \leq 2^k-1$ and let
$A_{r,k}=\{\alpha \in I^{\mathbb{N}}\colon \sum_{i=1}^{k}\alpha_i 2^{i-1} =r
\}$ be a~cylinder in~$I^{\mathbb{N}}$. Let $\mathcal{P}_{r,k}[\nu]$ denote
the normalized projection to~$I^{\mathbb{Z}}$ of the restriction of the
measure~$\nu$ to the set $I^{\mathbb{Z}}\times A_{r,k}$, \,$\theta_k[\nu]$~the
measure $\mathcal{P}_{0,k}[\nu]$, and~$\Theta[\nu]$ the sequence of measures
$(\theta_k[\nu])_{k \geq 0}$.
\end{definition}

A~slightly more complicated example of a~measure in~$\mathbb{M}$ is given
by the following construction, and we call it a~\textit{measure of periodic type}.

\begin{definition}
\label{d9} 
We denote by~$\mathcal{M}_k$, $k \geq 0$, the set of all Borel
probability measures on~$I^{\mathbb{Z}}$ that are invariant
under~$\mathbb{S}^{2^k}$, and by~$\operatorname{Ex}(\mathcal{M}_k)$ the set
of extreme points in~$\mathcal{M}_k$.
\end{definition}

\begin{remark}
\label{r6} 
For every $k \geq 0$ the set~$\mathcal{M}_k$ with the topology
of weak convergence of measures is a~Poulsen simplex, that is, the set
$\operatorname{Ex}(\mathcal{M}_k)$ of its extreme points is dense
in~$\mathcal{M}_k$.
\end{remark}

\begin{definition}
\label{d10} 
Let $k\geq 0$ and $\eta \in \mathcal{M}_k$. We call the measure
\begin{equation}
\label{eq13}
D_k[\eta] = \sum_{i=0}^{2^k-1} \mathbb{S}^i \eta \times
(\chi_{\rule{0pt}{7pt}A_{i,
k}} m)
\end{equation}
on~$I^{\mathbb{Z}}\times I^{\mathbb{N}}$ a~\textit{measure of periodic type\/~$k$
with base\/~$\eta$}.
\end{definition}

We note that the measures of periodic type~$0$ are simply direct products.

\begin{remark}
\label{r7} 
If~$k \geq 0$ and~$\eta \in \mathcal{M}_k$, then $D_k[\eta] \in
\mathbb{M}$.
\end{remark}

\begin{remark}
\label{r8} 
When $0\leq k \leq n$ and $\eta \in \mathcal{M}_k$ the equation $
D_{n}[\eta] = D_{k}[\eta]$ holds.
\end{remark}

\begin{definition}
\label{d11} 
We say that a~measure~$\nu$ on~$I^{\mathbb{Z}}\times
I^{\mathbb{N}}$ \textit{is a~measure of periodic type~$k$} if~$k$ is the least
possible number for which there is a~measure $\eta \in \mathcal{M}_k$ such that
$\nu = D_k[\eta]$. We denote by~$\mathbb{M}_k$ the set of all measures
of periodic type not greater than~$k$ on~$I^{\mathbb{Z}}\times I^{\mathbb{N}}$ and
by~$\operatorname{Ex}(\mathbb{M}_{k})$ the set of extreme points
of~$\mathbb{M}_k$.

We say that a~measure~$\nu$ is a~\textit{measure of a~periodic type} if it is
a~measure of periodic type~$k$ for some $k \geq 0$.

A~measure $\nu \in \mathbb{M}$ which is not a~measure of periodic type is said
to be \textit{aperiodic}. We denote by~$\mathbb{M}_\infty$ the set of aperiodic
measures.
\end{definition}

\begin{remark}
\label{r9} 
The set~$\mathbb{M}$ of invariant measures on~$I^{\mathbb{Z}}\times
I^{\mathbb{N}}$ can be decomposed into the disjoint union of the set
of aperiodic measures and the sets of measures of periodic type~$k$, $k \geq 0$:
$$
\mathbb{M} = \mathbb{M}_\infty \cup \mathbb{M}_{0}\cup
\biggl(\bigcup_{k\geq 1} (\mathbb{M}_{k}\setminus \mathbb{M}_{k-1})\biggr).
$$

For every $k \geq 0$ the map~$D_k$ is an affine homeomorphism between the
sets~$\mathcal{M}_{k}$ and~$\mathbb{M}_k$, and hence a~bijection between
$\operatorname{Ex}(\mathcal{M}_{k})$ and $\operatorname{Ex}(\mathbb{M}_{k})$.
\end{remark}

\subsection{Approximation by measures of periodic type}
\label{ss5.2}
We proceed with the study of measures in~$\mathbb{M}$ of general
form.

\begin{remark}
\label{r10} 
For every measure $\nu \in \mathbb{M}$ the family of measures
$\mathcal{P}_{r,k}[\nu]$, \,$k \geq 0$, \,$0\leq r <2^k$, on~$I^{\mathbb{Z}}$
corresponding to~$\nu$ satisfies the following relations:
\begin{gather}
\mathbb{S}\mathcal{P}[\nu]=\mathcal{P}[\nu]; \qquad
\mathbb{S}^{2^k}\mathcal{P}_{r,k}[\nu] = \mathcal{P}_{r,k}[\nu];
\label{eq14}
\\[2pt]
\mathbb{S}\mathcal{P}_{r,k}[\nu]= \mathcal{P}_{r+1,k}[\nu],
\quad 0\leq r <2^k-1; \qquad
\mathbb{S}\mathcal{P}_{2^k-1,k}[\nu] = \mathcal{P}_{0,k}[\nu];
\label{eq15}
\\
\label{eq16}
\mathcal{P}_{r,k}[\nu] =
\frac{1}{2}\bigl(\mathcal{P}_{r,k+1}[\nu]+\mathcal{P}_{r+2^k,k+1}[\nu]\bigr)=
\frac{1}{2}\bigl(\mathcal{P}_{r,k+1}[\nu]+\mathbb{S}^{2^k}\mathcal{P}_{r,k+1}[\nu]\bigr).
\end{gather}
\end{remark}

\begin{corollary}
\label{c1} 
Let\/ $\nu \in \mathbb{M}$. The sequence of measures\/
$\theta_{k}[\nu]$, $k \geq 0$, uniquely determines the whole family of measures\/
$\mathcal{P}_{r,k}[\nu]$, and hence also the measure\/~$\nu$. Moreover, the
following relations hold:
\begin{equation}
\label{eq17}
\mathbb{S}^{2^k}\theta_k=\theta_k, \qquad
\theta_{k} = \frac{1}{2}\bigl(\theta_{k+1}+\mathbb{S}^{2^k}\theta_{k+1}\bigr),
\quad k \geq 0.
\end{equation}
Further, if a~sequence of measures\/~$\theta_k$, $k \,{\geq}\, 0$,
on\/~$I^{\mathbb{Z}}$ satisfies relations~\eqref{eq17}, then there is
a~unique measure\/ $\nu \in \mathbb{M}$ for which\/ $\theta_k=\theta_k[\nu]$, 
\,$k \geq 0$.
\end{corollary}

\begin{definition}
\label{d12} 
We denote by~$\operatorname{proj}_k$ the projection of the
set~$\mathcal{M}_{k+1}$ to~$\mathcal{M}_{k}$ given by the formula
$$
\operatorname{proj}_k \theta = \frac{1}{2}\bigl(\theta+\mathbb{S}^{2^k}\theta\bigr),
\qquad \theta \in \mathcal{M}_{k+1}.
$$
\end{definition}

\begin{lemma}
\label{l2} 
The map\/~$\Theta$ is an affine homeomorphism of the space\/~$\mathbb{M}$
of measures with the weak topology onto the projective limit\/
$\varprojlim (\mathcal{M}_k,\operatorname{proj}_k)$.
\end{lemma}

\begin{definition}
\label{d13} 
We say that a~sequence of measures $(\theta_k)_{k \geq 0}$
on~$I^{\mathbb{Z}}$ \textit{stabilizes at a~moment~$n$}, $n \geq 0$,
if~$\theta_k=\theta_{k+1}$ for $k \geq n$ and $\theta_{n-1}\ne \theta_n$.
\end{definition}

\begin{remark}
\label{r11} 
For every $k \geq 0$ the map~$\Theta$ is a~bijection between the
set $\mathbb{M}_k\setminus \mathbb{M}_{k-1}$ of measures of periodic type~$k$
and the sequences stabilizing at the moment~$k$. Here the following identity
holds for $\eta \in \mathcal{M}_k$:
$$
\theta_n[D_k[\eta]] = \eta, \qquad n \geq k.
$$
\end{remark}

\begin{corollary}
\label{c2} 
The measures of periodic type are dense in\/~$\mathbb{M}$.
\end{corollary}

\begin{proof}
It can readily be seen that every measure $\nu \in \mathbb{M}$ is approximated
in the weak topology by the sequence of measures $D_k[\theta_k[\nu]]$, $k \to
\infty$. 
\end{proof}

As noted in Remark~\ref{r6}, the set~$\mathbb{M}_k$ is a~Poulsen simplex for
every $k \geq 0$, that~is, its Choquet boundary, the
set~$\operatorname{Ex}(\mathbb{M}_k)$, is dense in~$\mathbb{M}_k$.
Corollary~\ref{c2} means that the set $\bigcup_{k \geq 0} \mathbb{M}_k$
of measures of periodic type is dense in the simplex~$\mathbb{M}$ of all invariant
measures on the space $I^{\mathbb{Z}}\times I^{\mathbb{N}}$. This implies that
the set $\bigcup_{k \geq 0} \operatorname{Ex}(\mathbb{M}_k)$ is dense
in~$\mathbb{M}$. However, some of the measures in~$\bigcup_{k \geq 0}
\operatorname{Ex}(\mathbb{M}_k)$ are not extreme points of~$\mathbb{M}$ 
(see~Lemma~\ref{l5} below). Thus, a~natural question arises: \textit{is~$\mathbb{M}$
a~Poulsen simplex?} At the~moment, the authors do not know the answer to this
question. It seems that the question of whether or not the Poulsen property is
preserved under certain operations over simplices of invariant measures has not
been studied.

\subsection({Ergodic invariant measures on the
space~\$I\000\136\{\000\134mathbb\{Z\}\}\000\134times
I\000\136\{\000\134mathbb\{N\}\}\$}){Ergodic invariant measures on the space
$I^{\mathbb{Z}}\times I^{\mathbb{N}}$}
\label{ss5.3}
Let us proceed with the
study of the set $\operatorname{Ex}(\mathbb{M})$, which is the set of measures
in~$\mathbb{M}$ that are ergodic with~respect to~$\mathbb{S}\times\mathbb{O}$.

The following lemma gives a~characterization of ergodic measures
in~$\mathbb{M}$ in terms of their partial projections~$\theta_k[\cdot]$, $k
\geq 0$.

\begin{lemma}
\label{l3} 
Let\/ $\nu \in \mathbb{M}$. Then the following conditions are
equivalent:

\vskip1pt

{\rm 1)}~$\nu \in \operatorname{Ex}(\mathbb{M})$;

\vskip1pt

{\rm 2)}~$\theta_k[\nu] \in \operatorname{Ex}(\mathcal{M}_k)$ for every\/ $k \geq 0$;

\vskip1pt

{\rm 3)}~$\theta_k[\nu] \in \operatorname{Ex}(\mathcal{M}_k)$ for some\/ $n\geq 0$ 
and every\/ $k \geq n$.
\end{lemma}

\begin{proof}
To show that~1) implies~2), let $k \geq 0$ and recall that the symbol
$A_{r,k}$ denotes the coordinate cylinders in~$I^{\mathbb{N}}$ (see
Definition~\ref{d8}). Let $B \subset I^{\mathbb{Z}}$ be an
$\mathbb{S}^{2^k}$-invariant set. Then the set
$$
\bigcup_{r=0}^{2^k-1} (\mathbb{S}^r B) \times A_{r,k}
$$
is $\mathbb{S}\times\mathbb{O}$-invariant. Hence, by the ergodicity of the
measure~$\nu$, the value of~$\nu$ on this set is zero or one. This implies that
$\theta_k[\nu](B)$ is also zero or one.

\vskip1pt

2) obviously implies~3). Now suppose that condition~3) holds. 
To prove~1), suppose the contrary, that $\nu$ is not ergodic. Then there are
distinct $\nu_1, \nu_2 \in \mathbb{M}$ for which $\nu_1+\nu_2=2\nu$. Then for
every $k\geq n$ we have $\theta_k[\nu_1]+\theta_k[\nu_2]=2\theta_k[\nu]$. Since
the measure~$\theta_k[\nu]$ is ergodic, we have the equation
$\theta_k[\nu_1]=\theta_k[\nu_2]$ for $k \geq n$. By relation~\eqref{eq17},
this also holds for all $k \geq\nobreak 0$. Then the measures~$\nu_1$ and~$\nu_2$
coincide, which contradicts the procedure of their construction. 
\end{proof}

\begin{corollary}
\label{c3} 
Let\/ $\nu \in \operatorname{Ex}(\mathbb{M})$. Then for every\/ $k\geq 0$ 
and distinct\/~$r_1$ and\/~$r_2$ in the set\/ $\{0,\dots, 2^k-1\}$, the
measures\/ $\mathcal{P}_{r_1,k}[\nu]$ and\/~$\mathcal{P}_{r_2,k}[\nu]$ either
coincide or are mutually singular.
\end{corollary}

\begin{proof}
Each of the measures $\mathcal{P}_{r,k}[\nu]$, $0\leq r \leq 2^k-1$, is ergodic
for the transformation~$\mathbb{S}^{2^k}$. Any two ergodic measures either
coincide or are mutually singular. 
\end{proof}

\subsubsection{Ergodic measures of periodic type}
\label{sss5.3.1}
We need the following simple lemma.

\begin{lemma}
\label{l4} 
Let\/~$T$ be an automorphism of a~Lebesgue space\/ $(X,\mu)$. If the
transformation\/~$T^2$ is ergodic on\/~$(X,\mu)$, then for every\/ $k\geq 1$ 
the transformation\/~$T^{2^k}$ is also ergodic on\/~$(X,\mu)$.
\end{lemma}

\begin{proof}
The ergodicity of the transformation~$T^2$ implies that of the
automorphism~$T$ itself, and also the fact that~$-1$ is not an eigenvalue
of the unitary operator~$U_T$ on $L^2(X,\mu)$ corresponding to the
automorphism~$T$. Since the eigenvalues of~$U_T$~form a~group,
this implies that the only one which is a~root
of degree~$2^k$ of~$1$ is the number~$1$. Therefore, $1$ is a~multiplicity-free
eigenvalue of the operator~$T^{2^k}$. 
\end{proof}

The following lemma gives a~description of the set of measures of periodic type
in~$\operatorname{Ex}(\mathbb{M})$.

\begin{lemma}
\label{l5} 
Let\/ $k \geq 0$ and\/ $\eta \in \mathcal{M}_k$. Then\/ $D_k[\eta] \in
\operatorname{Ex}(\mathbb{M})$ if and only if\/~$\eta \in
\operatorname{Ex}(\mathcal{M}_{k+1})$.
\end{lemma}

\begin{proof}
It follows from Remark~\ref{r11} that the equation 
$\theta_{n}[D_k[\eta]] \,{=}\, \eta$ holds for $n\geq\nobreak k$.

If the measure $D_k[\eta]$ is ergodic for $\mathbb{S}\times\mathbb{O}$, then,
by Lemma~\ref{l3}, the measure $\eta = \theta_{k+1}[D_k[\eta]]$ is ergodic for
$\mathbb{S}^{2^{k+1}}$.

If the measure~$\eta$ is ergodic for $\mathbb{S}^{2^{k+1}}$\!, then,
by Lemma~\ref{l4}, it is ergodic for $\mathbb{S}^{2^{n}}$\!, $n\,{>}\,k$. Hence, the
measure $D_k[\eta]$ satisfies condition~3) of Lemma~\ref{l3}. This implies
that $D_k[\eta]$ is ergodic for $\mathbb{S}\times\mathbb{O}$ and completes the
proof of the lemma. 
\end{proof}

\begin{corollary}
\label{c4} 
Let\/ $k\geq 0$. Then (here and below we take\/ 
$\mathbb{M}_{-1}=\varnothing$)
\begin{equation}
\label{eq18}
\operatorname{Ex}(\mathbb{M})\cap \mathbb{M}_k\setminus
\mathbb{M}_{k-1} = \operatorname{Ex}(\mathbb{M}_k)\cap
\operatorname{Ex}(\mathbb{M}_{k+1})\setminus
\operatorname{Ex}(\mathbb{M}_{k-1}).
\end{equation}
\end{corollary}

\begin{proof}
We first prove that the left-hand side of equation~\eqref{eq18} is contained
on the right-hand side. If~$\nu \in \operatorname{Ex}(\mathbb{M})\cap
\mathbb{M}_k\setminus \mathbb{M}_{k-1}$, then, obviously, $\nu \in
\operatorname{Ex}(\mathbb{M}_{k})$. Further, there is a~measure $\eta \in
\mathcal{M}_k\setminus\mathcal{M}_{k-1}$ for which $\nu =D_k[\eta]$. By
Lemma~\ref{l5}, since $\nu \in \operatorname{Ex}(\mathbb{M})$, it follows that
$\eta \in \operatorname{Ex}(\mathcal{M}_{k+1})$. Hence, $\nu \in
\operatorname{Ex}(\mathbb{M}_{k+1})$. Here $\eta \notin \mathcal{M}_{k-1}$, and
therefore $\nu \notin \operatorname{Ex}(\mathbb{M}_{k-1})$.

To prove the reverse, let $\nu \in \operatorname{Ex}(\mathbb{M}_k)
\cap \operatorname{Ex}(\mathbb{M}_{k+1})\setminus
\operatorname{Ex}(\mathbb{M}_{k-1})$. Then $\nu = D_k[\eta]$ for some $\eta \in
\mathbb{M}_k\cap \operatorname{Ex}(\mathcal{M}_{k+1})$. It follows from
Lemma~\ref{l5} that $\nu \in \operatorname{Ex}(\mathbb{M})$. Moreover, if~$\nu
\in \operatorname{Ex}(\mathbb{M})$, then, obviously, $\nu \in
\operatorname{Ex}(\mathbb{M}_{k-1})$, that is, we arrive at a~contradiction.
Therefore $\nu \in \operatorname{Ex}(\mathbb{M})\cap \mathbb{M}_k\setminus
\mathbb{M}_{k-1}$. 
\end{proof}

Moreover, the following relation is obvious:
$$
\operatorname{Ex}(\mathbb{M})\cap
\mathbb{M}_\infty = \operatorname{Ex}(\mathbb{M}_\infty).
$$

\begin{corollary}
\label{c5}
The set of ergodic measures in\/~$\mathbb{M}$ admits the following grading:
$$
\operatorname{Ex}(\mathbb{M}) = \operatorname{Ex}(\mathbb{M}_\infty) \cup
\biggl(\bigcup_{k\geq 0} (\operatorname{Ex}(\mathbb{M}_k)\cap
\operatorname{Ex}(\mathbb{M}_{k+1})\setminus
\operatorname{Ex}(\mathbb{M}_{k-1}))\biggr).
$$
\end{corollary}

\subsubsection{Aperiodic ergodic measures}
\label{sss5.3.2}

\begin{theorem}
\label{t3} 
Let\/ $\nu \in \operatorname{Ex}(\mathbb{M}_\infty)$. Then for
every\/ $k \geq 0$ and distinct\/~$r_1$ and\/~$r_2$ in the set\/ $\{0,\dots,
2^k-1\}$, the measures\/ $\mathcal{P}_{r_1,k}[\nu]$ and\/~$\mathcal{P}_{r_2,k}[\nu]$
are mutually singular.

\end{theorem}

\begin{proof}
Assume that for some $k \geq 0$, \,$\mathcal{P}_{r_1,k}[\nu]$ and
$\mathcal{P}_{r_2,k}[\nu]$ are not mutually singular and let~$k$
be minimal among such numbers. We may assume without loss of generality
that $r_1<r_2$. By Corollary~\ref{c3}, the measures $\mathcal{P}_{r_1,k}[\nu]$
and~$\mathcal{P}_{r_2,k}[\nu]$ coincide. Since~$k$~is~minimal, all the measures
$\mathcal{P}_{r,k-1}[\nu]$, $0\leq r <\nobreak 2^{k-1}$, are pairwise singular. 
Using this argument and relation~\eqref{eq16}, we see that $r_2-r_1 = 2^{k-1}$.
Hence, $\mathcal{P}_{r_1,k-1}[\nu]$ coincides with 
$\mathcal{P}_{r_1,k}[\nu]$ and is ergodic for the
transformation~$\mathbb{S}^{2^k}$. Thus, the measure $\theta_{k-1}[\nu]$ is
also ergodic for~$\mathbb{S}^{2^k}$. By Lemma~\ref{l4}, 
$\theta_{k-1}[\nu]$ is ergodic for~$\mathbb{S}^{2^{n}}$ when $n \geq\nobreak k$. 
On the other hand, when $n>k$, by formula~\eqref{eq17}, the measure
~$\theta_{k-1}[\nu]$ is the arithmetic mean of the shifts of the measure
$\theta_{n}[\nu]$ which are invariant under~$\mathbb{S}^{2^{n}}$. Hence,
$\theta_{n}[\nu]=\theta_{k-1}[\nu]$ when $n\geq\nobreak k$. Hence,
$\nu=D_k[\theta_{k-1}[\nu]]$, that~is, $\nu$~is a~measure
of periodic type, which contradicts the assumption. 
\end{proof}

\begin{definition}
\label{d14}
For $n \geq 0$ we denote the number $e^{2^{1-n}\pi i }$ by~$\varrho_n$
(it is the $2^n$th root of unity with the least positive argument).
\end{definition}

\begin{corollary}[{\rm (properties of aperiodic ergodic measures)}]
\label{c6} 
Let\/ $\nu$ be an aperiodic measure in\/~$\operatorname{Ex}(\mathbb{M})$. Then

\vskip0.5pt

{\rm 1)}~the measures\/ $\theta_k[\nu]$ and\/~$\mathbb{S}^r\theta_k[\nu]$ are
mutually singular for\/ $k\geq 1$ and\/ $1<r<2^{k}-1$;

\vskip0.5pt

{\rm 2)}~(delta-type property) for\/ $\theta_0[\nu]$-almost all\/ $w \in
I^{\mathbb{Z}}$ the conditional measures in the sections\/ $\{w\}\times
I^{\mathbb{N}}$ are delta measures;

\vskip0.5pt

{\rm 3)}~the projection of\/~$I^{\mathbb{Z}}\times I^{\mathbb{N}}$
to~$I^{\mathbb{Z}}$ is an isomorphism of the dynamical systems\/
$(I^{\mathbb{Z}}\times I^{\mathbb{N}},\,\nu,\,\mathbb{S}\times\mathbb{O})$ and\/
$(I^{\mathbb{Z}},\,\theta_0[\nu],\,\mathbb{S})$;

\vskip0.5pt

{\rm 4)}~the numbers\/~$\varrho_n$, $n \geq 0$, are the eigenvalues for the
automorphism\/~$\mathbb{S}$ on the space\/ $(I^{\mathbb{Z}}, \theta_0[\nu])$;

\vskip0.5pt

{\rm 5)}~the dynamical system\/ $(I^{\mathbb{Z}},\,\theta_0[\nu],\,\mathbb{S})$
has a~quotient isomorphic to the odometer\/ $(I^{\mathbb{N}},m,\mathbb{O})$.
\end{corollary}

The following natural question arises: which measures on~$I^{\mathbb{Z}}$ can be
projections of aperiodic measures in~$\operatorname{Ex}(\mathbb{M})$? By
Lemma~\ref{l3}, these measures are ergodic with respect to the
shift~$\mathbb{S}$. It turns out that the condition in part~5)
of Corollary~\ref{c6} is not only necessary but also sufficient.

\begin{lemma}
\label{l6} 
Let\/ $\eta \in \operatorname{Ex}(\mathcal{M}_0)$ and let the
shift\/~$\mathbb{S}$ on the space\/ $(I^{\mathbb{Z}},\eta)$ have a~quotient
isomorphic to the odometer. Then there is an aperiodic measure\/ $\nu \in
\operatorname{Ex}(\mathbb{M})$ for which\/ $\theta_0[\nu]=\eta$. Moreover, the
set of all measures\/~$\nu$ of this kind is naturally parametrized by the points\/
$\alpha \in I^{\mathbb{N}}$ (see Lemma~\ref{l7}).
\end{lemma}

\begin{proof}
For every $n \geq 0$ we can find a~unique (up to constant) eigenfunction~$f_n$
corresponding to the eigenvalue~$\varrho_n$:
$$
\mathbb{S} f_n = \varrho_n f_n.
$$
Multiplying these functions by suitable constants if necessary, we may assume
that $f_0=\nobreak1$ and $f_{n}=f_{n+1}^2$, $n \geq 0$. Then, almost everywhere with
respect to the measure~$\eta$, the values of the function~$f_n$ coincide with
the powers $\varrho_n^{r}$, $0\leq r< 2^n-1$, $n \geq 0$. Consider the
level sets of the functions~$f_n$:
$$
B(r,n) = \{w \in I^{\mathbb{Z}}\colon f_n(w) = \varrho_n^{r}\}, \qquad 0 \leq
n, \quad 0\leq r< 2^n-1.
$$
Obviously, for every fixed $n \,{\geq}\, 1$ the map~$\mathbb{S}$ permutes
the sets $B(r,n)$, $r\,{=}\,0,\dots, 2^n\!{-}1$, cyclically and therefore $\eta(B(r,n))=1/2^n$.

Let $\alpha \in I^{\mathbb{N}}$, $\alpha = (\alpha_k)_{k \geq 1}$. We write
$r(n,\alpha) = \sum_{k=0}^{n-1}2^{k}\alpha_{k+1}$, $n \geq 0$. For $n\geq 0$ we
define a~measure $\theta_n^{(\alpha)}$ as the normalized restriction of the
measure $\eta$ to the set $\mathcal{B}(n,\alpha) = B\bigl(r(n,\alpha),
n\bigr)$:
$$
\theta_n^{(\alpha)} = 2^n \eta\big|_{{\mathcal{B}(n,\alpha)}}, \qquad n \geq 0.
$$

It is clear that the measure $\theta_n^{(\alpha)}$ is invariant under
$\mathbb{S}^{2^n}$. It can readily be seen that the sequence of measures
$(\theta_n^{(\alpha)})_{n\geq 0}$ thus constructed satisfies 
relation~\eqref{eq17}. This follows immediately from the fact that the sets
$\mathcal{B}(n,\alpha)$ and~$\mathcal{B}(n+1,\alpha)$ are connected by the
equation
\begin{align*}
&\mathcal{B}(n,\alpha) =
\bigl\{f_{n+1}^2 = \varrho_n^{r(n,\alpha)}\bigr\} =
\bigl\{f_{n+1} = \varrho_{n+1}^{r(n,\alpha)}\bigr\}\cup
\bigl\{f_{n+1} = -\varrho_{n+1}^{r(n,\alpha)}\bigr\}
\\[2pt]
&\qquad=B\bigl(r(n,\alpha),n+1\bigr)\cup B\bigl(2^n+r(n,\alpha),n+1\bigr) =
\mathcal{B}(n+1,\alpha) \cup
\mathbb{S}^{2^n}\mathcal{B}(n+1,\alpha).
\end{align*}
By Corollary~\ref{c1}, the sequence $(\theta_n^{(\alpha)})_{n\geq 0}$ defines
a~measure $\nu^{(\alpha)}[\eta] \in\mathbb{M}$ for which
$\theta_n[\nu^{(\alpha)}[\eta]]=\theta_n^{(\alpha)}$, \,$n \geq 0$. In
particular,
$$
\theta_0[\nu^{(\alpha)}[\eta]] = \theta_0^{(\alpha)}=\eta.
$$

It remains to prove that the measure $\nu^{(\alpha)}[\eta]$ thus constructed is
ergodic. By Lemma~\ref{l3}, its ergodicity is equivalent to the condition 
that the measure $\theta_n^{(\alpha)}$ is ergodic under $\mathbb{S}^{2^n}$ 
for every $n \geq 0$. We now prove this fact.

\vskip1pt

Let a~set $B \subset I^{\mathbb{Z}}$ be invariant under $\mathbb{S}^{2^n}$.
Then the function
$$
f=\sum_{j=0}^{2^n-1} \varrho_n^j \chi_{\rule{0pt}{7pt}\mathbb{S}^j B}
$$
satisfies the equation $\mathbb{S} f {\kern1pt}{=}{\kern1pt} \varrho_n f$.\ 
All the eigenvalues of the operator $\mathbb{S}$ 
are multiplicity-\allowbreak free, and therefore $f = C f_n$ almost
everywhere with respect to the measure~~$\eta$ for some constant $C \in
\mathbb{C}$. This implies that either $\eta(B)=\nobreak0$ or the set~$B$ coincides with
one of the sets $B(r,n)$, $0\leq r \leq 2^n-1$. In particular, the set
$\mathcal{B}(n,\alpha)$ contains no non-trivial $\mathbb{S}^{2^n}$\!-invariant
subsets. Hence, the measure $\theta_n^{(\alpha)}$ is ergodic for
$\mathbb{S}^{2^n}$. This completes the proof of Lemma~\ref{l6}. 
\end{proof}

\begin{lemma}
\label{l7} 
The family of measures\/ $\nu^{(\alpha)}[\eta]$, $\alpha \in
I^{\mathbb{N}}$, constructed in the proof of Lemma~\ref{l6} describes all the
measures in\/~$\operatorname{Ex}(\mathbb{M})$ with the given projection\/~$\eta$.
In other words, if\/~$\nu \in \mathbb{M}$ and\/ $\theta_0[\nu]=\eta$, then\/ 
$\nu = \nu^{(\alpha)}[\eta]$ for some\/ $\alpha \in I^{\mathbb{N}}$.
\end{lemma}

Before passing to the proof of this lemma, we discuss a~spectral property
of measures of periodic type.

\begin{lemma}
\label{l8} 
Let\/ $k\geq0$, \,$\eta\in\mathcal{M}_k$ and\/
$D_k[\eta]\in\operatorname{Ex}(\mathbb{M})$. Then the
numbers\/~$\varrho_n$, $n>\nobreak k$, are not eigenvalues of the operator\/~$\mathbb{S}$
on the space\/ $(I^{\mathbb{Z}}, \theta_0[D_k[\eta]])$.
\end{lemma}

\begin{proof}
Let
\begin{equation}
\label{eq19}
\mathbb{S} f = \varrho_n f
\end{equation}
almost everywhere with respect to the measure $\theta_0=\theta_0[D_k[\eta]]$.
We claim that $f=\nobreak0$~almost everywhere with respect to the
measure~$\theta_0$. It is clear that $\mathbb{S}^{2^n}f =\nobreak f$ almost 
everywhere with respect to the measures $\mathbb{S}^j\eta$, \,$0\leq j \leq 2^k-1$. 
By Lemmas~\ref {l5} and~\ref{l4}, the measures~$\mathbb{S}^j\eta$, 
\,$0\leq j \leq 2^k-1$, 
are ergodic for $\mathbb{S}^{2^n}\!$. Hence, the function~$f$ is 
constant almost everywhere with respect to each of them. It follows that 
$f$ takes at most $2^k$ distinct values on some subset of full measure
with respect to~$\theta_0$. On the other hand, by 
relation~\eqref{eq19}, if~$f$~is not almost everywhere zero, then
it must take at least $2^n$ different values on sets of positive measure.
Hence, $f=\nobreak0$ almost everywhere with respect to~$\theta_0$. 
This completes the proof of Lemma~\ref{l8}. 
\end{proof}

\begin{proof}[of Lemma~\ref{l7}]
We use the notation in the proof of Lemma~\ref{l6}. Let $\nu \in
\mathbb{M}$ be ergodic and let $\theta_0[\nu]=\eta$. It follows from
Lemma~\ref{l8} that the measure~$\nu$ cannot~be~of periodic type. 
By Corollary~\ref{c6}, the measures $\mathcal{P}_{r,n}[\nu]$,
$0\leq r \leq 2^n-1$, are pairwise mutually singular for a~fixed $n \geq\nobreak 0$. 
Moreover, by the relations in Remark~\ref{r10}, the following equation holds:
$$
\eta = \frac{1}{2^n} \sum_{r=0}^{2^n-1} \mathcal{P}_{r,n}[\nu].
$$
Let a~set $B_n \subset I^{\mathbb{Z}}$ be such that $\mathcal{P}_{0,n}[\nu] =
2^n \eta\big|_{B_n}$. Then $\mathcal{P}_{r,n} = 2^n \eta\big|_{\mathbb{S}^r
B_n}$, \,$0\leq r \leq 2^n-1$, and the sets $\mathbb{S}^{r}B_n$ are pairwise
disjoint.

The function
$$
g_n = \sum_{r=0}^{2^n-1} \varrho_n^r \chi_{\rule{0pt}{7pt}\mathbb{S}^{r}B_n}
$$
satisfies the relation $\mathbb{S} g_n = \varrho_n g_n$. Therefore, $g_n = c_n
f_n$ for some constant $c_n \in \mathbb{C}$. This implies that $B_n = B(r_n,n)$
for some $r_n$, $0\leq r_n \leq 2^n-1$, and $c_n = \varrho_n^{-r_n}$.

\vskip1pt

{\advance\baselineskip by 0.7pt

It remains to find out how the numbers~$r_n$ and~$r_{n+1}$ are related. We note
that the measure $\mathcal{P}_{0,n}$ is the semi-sum of the measures
$\mathcal{P}_{0,n+1}$ and $\mathcal{P}_{2^n,n+1}$. Hence, $B_{n+1}$ is
a~subset of~$B_n$. Thus, $B(r_{n+1},n+1)\,{\subset}\,B(r_n,n)$. By the
definition of these sets, this means that $\varrho_{n+1}^{2r_{n+1}}\,{=}\,
\varrho_{n}^{r_{n}}$, that is, $\smash{(r_{n+1}-r_n) \,\vdots\, 2^n}$. Hence, either
$r_{n+1}=r_n$ or $r_{n+1}=r_n+2^n$. We set $\alpha_{n+1}=2^{-n}(r_{n+1}-r_n)$,
\,$n \geq\nobreak 0$. Then $r_n=r(n,\alpha)$ and $B_n= \mathcal{B}(n,\alpha)$, which
implies that $\theta_n[\nu]=\theta_n^{(\alpha)}$ and
$\nu=\nu^{(\alpha)}[\eta]$. This completes the proof of Lemma~\ref{l7}. 

}
\end{proof}

\subsubsection{Description of all ergodic measures}
\label{sss5.3.3}
Finally, we have obtained the following description of the ergodic
measures in~$\mathbb{M}$.

\begin{theorem}
\label{t4} 
The ergodic measures in the set\/~$\mathbb{M}$ are divided into two
disjoint classes.

{\rm 1)}~Ergodic measures of periodic type\/~$k$, $k \geq 0$, of the form
$$
D_k[\eta], \quad \text{where\/}\quad \eta \in \mathcal{M}_k\cap
\operatorname{Ex}(\mathcal{M}_{k+1}).
$$
The projections of ergodic measures of periodic type to the space\/
$I^{\mathbb{Z}}$ admit the following description: a~measure\/ $\widetilde\eta$
on\/~$I^{\mathbb{Z}}$ which is ergodic with respect to the shift\/~$\mathbb{S}$ 
is the projection of some ergodic measure of periodic type\/~$k$
on\/~$I^{\mathbb{Z}}\times I^{\mathbb{N}}$,  
\,$k \geq\nobreak 0$, if and only if 
the operator\/ $\mathbb{S}$ on\/~$(I^{\mathbb{Z}},\eta)$ has a~quotient isomorphic 
to the shift on the set of\/~$2^n$ points with uniform measure with\/ $n=k$ 
rather than\/ $n=k+1$.

{\rm 2)}~Ergodic aperiodic measures of the form\/ $\nu^{(\alpha)}[\eta]$,
where the measure\/~$\eta$ on~the space\/ $I^{\mathbb{Z}}$ is invariant 
and ergodic with respect to the shift\/~$\mathbb{S}$ and, moreover, 
the~operator\/~$\mathbb{S}$ on\/~$(I^{\mathbb{Z}},\eta)$ has a~quotient 
isomorphic to the odometer.\footnote{
The authors are grateful to the referee for indicating
another possible proof of this theorem using the general technique
of (Furstenberg--Zimmer) automorphism extensions. While our proof is 
more specific, the alternative may possibly generalize in a~simpler way 
to problems of describing the invariant measures for adic transformations.}
\end{theorem}

\begin{remark}
\label{r12} 
The ergodic measures in~$\mathbb{M}$ of periodic type~$0$ are
direct products of the form $\eta\times m$, where $\eta$ stands for an
$\mathbb{S}$-invariant measure on~$I^{\mathbb{Z}}$ which is ergodic 
with respect to the transformation~$\mathbb{S}^2$.
\end{remark}

The properties of ergodic aperiodic measures were described in detail
in Corollary~\ref{c6}.

\subsection({Additional information about the ergodic measures
on~\$I\000\136\{\000\134mathbb\{Z\}\}\000\134times
I\000\136\{\000\134mathbb\{N\}\}\$: the structure of conditional
measures}){Additional information about the ergodic measures
on~$I^{\mathbb{Z}}\times I^{\mathbb{N}}$: 
the structure of conditional measures}
\label{ss5.4}
Let an ergodic automorphism~$T$ of the~Lebesgue space $(X,\mu)$
have an invariant measurable partition~$\varsigma$. In this case, 
the automorphism~$T$ has a~quotient automorphism~$T_{\varsigma}$ 
acting on the quotient space~$X_{\varsigma}$ with the quotient 
measure~$\mu_\varsigma$. Then one says that 
\textit{$T$ is a~skew product over the automorphism\/~$T_{\varsigma}$
of the space\/~$X_{\varsigma}$} with some system of automorphisms of the fibres
(see, for example, \cite{28}). In this case, it is usually assumed that, because
of the ergodic property, the space $(X,\mu)$ is naturally isomorphic to the direct
product of the space $(X_{\varsigma}, \mu_{\varsigma})$ and the `generic
fibre' with the generic measure of the fibre, that is, a~generic conditional
measure on the elements of the partition $\varsigma$. However, especially
if the space~$X$ is equipped with the topology of a~separable metric space, it is 
more convenient not to make a~uniformization of the fibres but to assume that the
fibres remain subsets of the original space~$X$ with conditional measures
as Borel measures on the fibres of the partition. Then the action of the
automorphism is not changed, and only the problem of explicitly calculating
the conditional measures remains. Usually, in Rokhlin theory, these
measures are computed using a~basis of the measurable partition~$\varsigma$,
and this is a~general method.\footnote{
We note that the family of conditional
(canonical) measures corresponds to the kernel of the operator of conditional
expectation (an orthogonal projection) onto the subalgebra corresponding to the
partition~$\varsigma$.} 
In our case, the basis of the partition is determined
by the structure of the quotient automorphism (the odometer), more precisely,
by its spectrum. Thus, almost all the conditional measures (which are pairwise
mutually singular) become limits of sequences of measures of the type
of ergodic averagings (see below). We do not use any information about the
conditional measures below. However, a~similar method can be used for other
skew products.

\begin{definition}
\label{d15} 
We say that a~probability measure~$\vartheta$ on~$I^{\mathbb{Z}}$
is \textit{averageable} if for any $k \geq 1$ there is a~measure $\vartheta_k$
on~$I^{\mathbb{Z}}$ which is the weak limit of the sequence of measures
$\vartheta_{k,n}$, $n\to \infty$, where
\begin{equation}
\label{eq20}
\vartheta_{k,n} = \frac{1}{n} \sum_{j=0}^{n-1} \mathbb{S}^{j 2^k}\vartheta,
\end{equation}
and the sequence of measures $\vartheta_{k}$ tends weakly to the
measure~$\vartheta$ as~$k\to \infty$.
\end{definition}

The following lemma is a~consequence of the ergodic theorem and the Lebesgue
differentiation theorem.

\begin{lemma}
\label{l9} 
If\/~$\nu \in \mathbb{M}$, then for almost all\/ $\alpha \in
I^{\mathbb{N}}$ the conditional measure on the fibre\/
$I^{\mathbb{Z}}\times\{\alpha\}$ is averageable as a~measure
on\/~$I^{\mathbb{Z}}$. Here the measure\/~$\nu$ is uniquely determined 
by the generic conditional measure.
\end{lemma}

\begin{proof}
Let $\nu[\alpha]$ be the conditional measure on a~fibre
$I^{\mathbb{Z}}\times\{\alpha\}$. Then, since the measure~$\nu$ is invariant
under the transformation $\mathbb{S}\times\mathbb{O}$, it follows
that for almost all $\alpha \in I^{\mathbb{N}}$ we have
$$
\mathbb{S}^{j} \nu[\alpha] = \nu[\mathbb{O}^j \alpha], \qquad j \in \mathbb{Z}.
$$

Let $\phi$ be the characteristic function of some cylindrical set
in~$I^{\mathbb{Z}}$. Consider the function~$\psi$ on~$I^{\mathbb{N}}$
defined almost everywhere by the formula
$$
\psi(\alpha) = \int_{I^{\mathbb{Z}}} \phi(w) \,d\nu[\alpha](w),
$$
that is, the integral of the function~$\phi$ over the fibre. Let $k \geq 0$.
Then for almost all $\alpha \in I^{\mathbb{N}}$ we have the equation
\begin{equation}
\label{eq21}
\frac{1}{n}\sum_{j=0}^{n-1} \psi\bigl(\mathbb{O}^{j 2^k} \alpha\bigr) =
\int_{I^{\mathbb{Z}}} \phi(w) \,d\nu[\alpha]_{k,n}(w),
\qquad n \geq 1,
\end{equation}
where the measure $\nu[\alpha]_{k,n}$ is the averaging (of the measure
$\nu[\alpha]$) defined by formula~\eqref{eq20}. The transformation
$\mathbb{O}^{2^k}$ is ergodic on each of the cylinders $A_{r,k}$, $0\leq r
\leq 2^k-1$, (see Definition~\ref{d8}) with the Lebesgue measure, and
therefore, by the ergodic theorem, for almost all $\alpha \in I^{\mathbb{N}}$
the left-hand side of equation~\eqref{eq21} converges as~$n\to \infty$
to the mean value of the function~$\psi$ over the cylindrical set~$A_{r,k}$
containing the point~$\alpha$, that is,
$$
\int_{I^{\mathbb{Z}}} \phi(w) \,d\nu[\alpha]_{k,n}(w) \xrightarrow{n \to
\infty} 2^k \int_{(w,\beta) \in I^{\mathbb{Z}}\times A_{r,k}} \phi(w) \,d\nu(w,\beta) =
\int_{I^{\mathbb{Z}}} \phi\,d\mathcal{P}_{r,k}[\nu]
$$
provided that $\alpha \in A_{r,k}$. Hence, for every $k \geq 0$ and almost
every $\alpha \in I^{\mathbb{N}}$ the measures $\nu[\alpha]_{k,n}$ converge
weakly to the measure $\mathcal{P}_{r(k,\alpha),k}[\nu]$, where the number
$r(k,\alpha)$ is chosen in such a~way that $\alpha \in A_{r(k,\alpha),k}$.

\vskip1pt

Further, applying the Lebesgue differentiation theorem, we see that for almost
every $\alpha \in I^{\mathbb{N}}$ the average value of the function~$\psi$ over
the cylindrical set $A_{r(k,\alpha),k}$ containing the point~$\alpha$ converges
to~$\psi(\alpha)$ as~$k \to \infty$. By choosing~$\phi$ to be a~countable family 
of cylinders, we get that for almost all $\alpha \in I^{\mathbb{N}}$ 
the measure $\mathcal{P}_{r(k,\alpha),k}[\nu]$ converges to~$\nu[\alpha]$. 
Hence, for almost all $\alpha \in I^{\mathbb{N}}$ the
measure $\nu[\alpha]$ is averageable.

Moreover, for almost every $\alpha \in I^{\mathbb{N}}$ 
\,$\nu[\alpha]$ uniquely defines the sequence of measures
$\mathcal{P}_{r(k,\alpha),k}[\nu]$, $k \geq 0$ (as the limits of the averaging),
and therefore it also determines the whole of the measure~$\nu$. This completes the
proof of Lemma~\ref{l9}. 
\end{proof}

\begin{remark}
\label{r13} 
If~$\nu \in \operatorname{Ex}(\mathbb{M})$, then for almost all
$\alpha \in I^{\mathbb{N}}$ the conditional measures in~the fibres
$I^{\mathbb{Z}}\times\{\alpha\}$ are extreme points of the set of averageable
measures on~$I^{\mathbb{Z}}$.

\end{remark}

\section{Evaluation of a~scaling sequence}
\label{s6}

\vskip1pt

In this section we recall the definition of the scaling sequence of an action~of 
a~group (in our case, $\mathbb{Z}$ or~$\mathcal{D}$) and also the definition
of the scaling sequence of~a~filtration. In~\cite{27}, the scaling sequences of the
action~$\kappa$ of~$\mathcal{D}$ and of the adic action of~$\mathbb{Z}$ on the
path space $\mathcal{T}(\mathrm{OP})$ of the graph of ordered
pairs~$\mathrm{OP}$ with special measures~$\mu^\sigma$ were calculated. 
It turned out that the classes of scaling sequences of these actions coincide. 
In this section, we describe another way of calculating a~scaling sequence 
for the action~$\kappa$ of~$\mathcal{D}$ on the space
$(\mathcal{T}(\mathrm{OP}),\mu^\sigma)$. Moreover, we find the class of scaling
sequences of the tail filtration $\xi = (\xi_n)_{n \geq 0}$ on the space
$(\mathcal{T}(\mathrm{OP}),\mu^\sigma)$ and show that it coincides with the
class of scaling sequences of the actions described previously.

\subsection{Recalling the definitions and properties}
\label{ss6.1}
We recall the part of the theory of admissible metrics we need 
(see, for example,~\cite{29},~\cite{30},~\cite{5},~\cite{6},~\cite{11},~\cite{31} 
and~\cite{27}).

\subsubsection({Admissible semimetrics and \$\000\134varepsilon\$-entropy})
{Admissible semimetrics and $\varepsilon$-entropy}
\label{sss6.1.1}

\begin{definition}
\label{d16} 
Let $(X,\mu)$ be a~Lebesgue space. A~semimetric $\rho$ on~$X$ (and
the triple $(X,\mu,\rho)$) is said to be \textit{admissible} if~$\rho$ is
measurable as a~function of two variables with respect to the measure~$\mu^2$
and there is a~subset of full measure on which $\rho$ is separable.
If~$\rho$ is a~metric on such a~subset, then it is called an
\textit{admissible metric}. The cone of admissible summable (with respect
to~$\mu^2$) semimetrics on~$(X,\mu)$ is denoted
by~$\mathcal{A}\mathrm{dm}(X,\mu)$.
\end{definition}

To work with admissible semimetrics, the following norm on the space
of functions of two variables is convenient. It has been referred to 
as an $m$-norm (see~\cite{11}).

\begin{definition}
\label{d17} 
For $f \in L^1(X^2,\mu^2)$ we write
\begin{multline*}
\|f\|_m = \inf\bigl\{\|\rho\|_{_{L^1(X^2,\mu^2)}}\colon |f|\leq \rho
\ \,\text{almost everywhere with respect to }\,\mu^2, 
\\
\text{ where } \,\rho\, \text{ is a~measurable semimetric
on } \,(X,\mu)\bigr\}.
\end{multline*}
\end{definition}

\begin{definition}
\label{d18} 
Let $\rho$ be a~measurable (as a~function of two variables)
semimetric on~$(X,\mu)$ and let $\varepsilon>\nobreak0$. By the
$\varepsilon$-\textit{entropy of the triple} $(X,\mu,\rho)$ we mean the number
$\mathbb{H}_\varepsilon(X,\mu,\rho)$ defined as the binary logarithm
of the smallest positive integer~$k$ for which the space~$X$ can be partitioned
into measurable sets $X_0, \dots, X_k$ in such a~way that the set~$X_0$ has small
measure, $\mu(X_0)<\varepsilon$, and the sets~~$X_j$, $1\leq j \leq k$, have
diameters less than~$\varepsilon$ in the semimetric~$\rho$. If there is no
positive integer~$k$ of this kind, then
$\mathbb{H}_\varepsilon(X,\mu,\rho)=+\infty$.
\end{definition}

\begin{remark}
\label{r14} 
A~measurable semimetic~$\rho$ on~$(X,\mu)$ is admissible if and only if
$\mathbb{H}_\varepsilon(X,\mu,\rho)<+\infty$ \,for every \,$\varepsilon>\nobreak0$.
\end{remark}

\begin{lemma}[{\rm (see~\cite{11} and \cite{13})}]
\label{l10} 
$\varepsilon$-entropy has the following properties.

\vskip1.5pt

{\rm 1)}~The\/ $\varepsilon$-entropy\/ $\mathbb{H}_\varepsilon(X,\mu,\rho)$
decreases with respect to\/~$\varepsilon$ and increases with respect 
to\/~$\rho$.

\vskip1.5pt

{\rm 2)}~If\/ $\rho \in \mathcal{A}\mathrm{dm}(X,\mu)$ and\/ $\int_{X^2}\rho
\,d\mu^2 < \varepsilon^2/2$, then\/ $\mathbb{H}_\varepsilon(X,\mu,\rho)=\nobreak0$.

{\rm 3)}~If\/ $\rho,\rho_1,\rho_2 \,{\in}\, \mathcal{A}\mathrm{dm}(X,\mu)$ and\/
$\rho\,{\leq}\, \rho_1+\rho_2$ almost everywhere with respect~to\/~$\mu^2$,
then
$$
\mathbb{H}_{4\varepsilon}(X,\mu, \rho) \leq \mathbb{H}_\varepsilon(X,\mu,
\rho_1)+\mathbb{H}_\varepsilon(X,\mu, \rho_2).
$$

\vskip1.5pt

{\rm 4)}~If\/ $\rho_1,\rho_2 \in \mathcal{A}\mathrm{dm}(X,\mu)$ and\/
$\|\rho_1-\rho_2\|_m<\varepsilon^2/32$, then
$$
\mathbb{H}_\varepsilon(X,\mu, \rho_1)\leq
\mathbb{H}_{\varepsilon/4}(X,\mu, \rho_2).
$$
\end{lemma}

\subsubsection{Scaling sequences of a~measure-preserving transformation}
\label{sss6.1.2}

\begin{definition}
\label{d19} 
Let $(X,\mu)$ be a~Lebesgue space and $T$ a~measure-preserving
transformation on~$(X,\mu)$. Let $\rho$ be a~measurable semimetric on~$(X,\mu)$.
By the \textit{averaging of the semimetric\/~$\rho$ under the action of\/~$T$
in\/~$n$ steps}, $n \geq 1$, we mean the semimetric
$$
T_{\mathrm{av}}^n \rho= \frac{1}{n}\sum_{j=0}^{n-1} T^j\rho.
$$
\end{definition}

\vskip2pt

In what follows, we use the symbol~$\asymp$ for two sequences of positive
numbers bounding each other with some constant:
$$
a_n \asymp b_n \quad \Longleftrightarrow \quad 0<\liminf_{n \to
\infty}\frac{a_n}{b_n}\leq \limsup_{n \to \infty}\frac{a_n}{b_n}<\infty.
$$

\begin{definition}
\label{d20} 
Let $T$ be a~measure-preserving transformation of a~Lebesgue space
$(X,\mu)$ and $\rho$ an admissible semimetric on~$(X,\mu)$. A~sequence
$h=(h_n)_{n \geq 1}$ of positive numbers is said to be 
\textit{scaling for\/~$\rho$} if the relation
$$
\mathbb{H}_\varepsilon(X,\mu,T_{\mathrm{av}}^n \rho)\asymp h_n,\qquad n \to
\infty,
$$
holds for a~sufficiently small $\varepsilon>\nobreak0$. The class of all scaling
sequences for~$\rho$ is denoted by~$\mathcal{H}(X,\mu,T,\rho)$.
\end{definition}

In~\cite{13} (see also the short communication \cite{12}), the conjecture
formulated by Vershik and claiming that a~scaling sequence does not depend
on the semimetric~$\rho$ was proved in a~wide class of semimetrics.

\begin{definition}
\label{d21} 
A~semimetric~$\rho$ on~$(X,\mu)$ is said to be (two-sided)
\textit{generating} for a~transformation~$T$ if there is a~subset $X_0\subset X$ 
of full measure such that for any two distinct points $x,y \in X_0$ there is an
$n \in \mathbb{Z}$ for which $\rho(T^nx,T^ny)>\nobreak0$.
\end{definition}

Obviously, a~measurable metric is a~generating semimetric.

\begin{theorem}[{\rm (see~\cite{13})}]
\label{t5} 
Let\/ $\rho_1, \rho_2 \in \mathcal{A}\mathrm{dm}(X,\mu)$ be generating
semimetrics for a~measure-\allowbreak preserving transformation\/~$T$. 
Then\/ $\mathcal{H}(X,\mu,T,\rho_1)=\mathcal{H}(X,\mu,T,\rho_2)$.
\end{theorem}

The proof of this theorem uses the following lemma.

\begin{lemma}
\label{l11} 
Let\/ $\rho_1,\rho_2 \in \mathcal{A}\mathrm{dm}(X,\mu)$ be two
admissible metrics. Then for any\/ $\varepsilon>\nobreak0$ there is 
a~set\/ $X_0\subset X$ such that\/ $\mu(X_0)>1-\varepsilon$ and the 
semimetrics\/ $\rho_1$ and\/~$\rho_2$ define the same topology on\/~$X_0$.
\end{lemma}

Theorem~\ref{t5} means that the class of scaling sequences is a~characteristic
of the transformation~$T$.

\begin{definition}
\label{d22} 
A~sequence~$h$ is said to be a~\textit{scaling entropy sequence}
of a~measure-\allowbreak preserving transformation~$T$ of a~Lebesgue 
space $(X,\mu)$ if~$h \in \mathcal{H}(X,\mu,T,\rho)$ 
for some (and then for every) generating semimetric
$\rho \in \mathcal{A}\mathrm{dm}(X,\mu)$. We denote the class of scaling
entropy sequences of~$T$ by~$\mathcal{H}(X,\mu,T)$.
\end{definition}

The class of scaling entropy sequences is obviously a~metric invariant.

\vskip1pt

It was proved in~\cite{20} that if the class of scaling sequences is non-empty,
then it contains an increasing subadditive sequence. It was proved
in~\cite{27} that there~is a~dynamical system with a~prescribed scaling entropy
sequence if this sequence is subadditive and increases. Thus, a~characteristic
of all possible scaling entropy sequences of measure-\allowbreak preserving transformations
was given. Examples of dynamical systems with a~given growth of scaling
sequences were given by the adic transformation~$\mathcal{A}$ of the path space
$\mathcal{T}(\mathrm{OP})$ of the graph~$\mathrm{OP}$ of ordered pairs with
measures~$\mu^\sigma$ for different sequences~$\sigma$. Namely, the following
theorem was proved.

\begin{theorem}[{\rm (see~\cite{27})}]
\label{t6} 
The sequence\/ $h_n = 2^{\sum_{i=0}^{n-1}\sigma_i}$, $n \geq 1$, is
a~scaling sequence of the adic transformation\/~$\mathcal{A}$ of the space\/
$(\mathcal{T}(\mathrm{OP}),\mu^\sigma)$.
\end{theorem}

\subsubsection{Scaling sequences of the action of a~group}
\label{sss6.1.3}
By analogy with the definition for~a~single measure-\allowbreak preserving
transformation, we can introduce the definition~of a~scaling sequence of the
action of a~group~$G$. To do this, it is necessary to distinguish some equipment 
in~$G$, namely, a~family of subsets~$G_n$, $n \in \mathbb{N}$, over which
the averaging is to be carried out.

\begin{definition}
\label{d23} 
Let $G$ be a~group of automorphisms of a~Lebesgue space $(X,\mu)$
with equipment $G_n\subset G$, $n \geq 1$, and let $\rho$ be an admissible
semimetric on~$(X,\mu)$. A~sequence $h=(h_n)_{n \geq 1}$ of positive numbers is
called a~\textit{scaling sequence for~$\rho$}~\,if the~relation
$$
\mathbb{H}_\varepsilon\bigl(X,\mu,T_{\mathrm{av}}^{G_n} \rho\bigr) \asymp
h_n, \qquad n \to \infty,
$$
holds for a~sufficiently small~$\varepsilon>\nobreak0$, where $T_{\mathrm{av}}^{G_n}
\rho = \frac{1}{|G_n|}\sum_{g \in G_n} \rho(g \, \cdot\,)$ is the averaging
of~$\rho$ with respect to the shifts in~$G_n$.

We denote the class of all scaling sequences for a~semimetric~$\rho$ with
respect to the action of~$G$ by~$\mathcal{H}(X,\mu,G,\rho)$.
\end{definition}

As in the case of a~single transformation, the scaling sequences of an action
of an equipped group~$G$ do not depend on the choice of the original metric.

\begin{theorem}[{\rm (see~\cite{27})}]
\label{t7}
If\/~$\rho_1, \rho_2 \in \mathcal{A}\mathrm{dm}(X,\mu)$ are metrics, then
$$
\mathcal{H}(X,\mu,G,\rho_1)=\mathcal{H}(X,\mu,G,\rho_2).
$$
\end{theorem}

To replace the metrics in this theorem by generating semimetrics, it is
necessary to impose some conditions on the sequence of sets~$G_n$
(see~\cite{27}). As in the case of a~single transformation, we introduce the
notion of a~\textit{scaling sequence of an action of an equipped group}; this is
a~scaling sequence of an arbitrary admissible metric. The class of these
sequences is a~metric invariant of the action of the equipped group~$G$ and is
denoted by~$\mathcal{H}(X,\mu,G)$.

The following assertion enables us to prove that some number sequence is
scaling by testing this on a~suitable sequence of semimetrics which need not be
generating.

\begin{lemma}
\label{l12}
Let a~sequence of semimetrics\/ $\rho_k\,{\in}\,\mathcal{A}\mathrm{dm}(X,\mu)$,
$k\,{\geq}\, 1$, together separate the points of a~subset of full measure. 
Let a~sequence\/~$h$ of positive numbers be such that\/ $h \in
\mathcal{H}(X,\mu,G,\rho_k)$, $k \geq 1$. Then\/ $h \in \mathcal{H}(X,\mu,G)$.
\end{lemma}

\begin{proof}
One can choose a~sequence of sufficiently small positive numbers~$c_k$, 
$k \geq\nobreak 1$, in such a~way that the function $\rho \,{=} \sum c_k \rho_k$ is
in~$\mathcal{A}\mathrm{dm}(X,\mu)$. Here~$\rho$ is an admissible metric. Using
Lemma~\ref{l10}, we can readily see that $h \in \mathcal{H}(X,\mu,G,\rho)$, and
therefore $h \in \mathcal{H}(X,\mu,G)$. 
\end{proof}

\subsubsection{Scaling sequences of a~filtration}
\label{sss6.1.4}
The notion of the scaled
entropy of a~filtration can be found in~\cite{7} and~\cite{8}, and also
in the doctoral thesis of the first author (see also~\cite{5}, \cite{10},
\cite{21} and~\cite{29}). We recall one of the possible definitions
of a~scaling sequence of a~filtration.

Let $\varsigma = (\varsigma_n)_{n\geq 0}$ be a~dyadic filtration on a~Lebesgue
space $(X,\mu)$, where $\varsigma_0$ is the partition into points. Every
element of the partition $\varsigma_n$, $n \geq 0$, is naturally equipped with
a~dyadic hierarchy, namely, this element consists of two elements of the
partition $\varsigma_{n-1}$, each of which also consists of two elements of the
partition $\varsigma_{n-2}$, and so on up to the partitioning into points. The
group $\mathcal{T}_n$ of automorphisms of the binary tree of height~$n$ (this
tree has $n+1$ floors of vertices and $n$ floors of edges) acts on the points
of every element of the partition $\varsigma_n$, and preserves the hierarchy 
(that~is, the elements of previous partitions).

Let $\rho$ be a~semimetric on~$(X,\mu)$. For every $n \geq 0$ we construct
a~semimetric $\mathcal{K}_n=\mathcal{K}_n[\rho]$ on the set of elements of the
partition~$\varsigma_n$ as follows. Let $c_1,c_2$ be two elements of the
partition~$\varsigma_n$, \,$c_i=\{x_{i,j}\colon j=1,\dots, 2^n\}$, \,$i=1,2$. 
We write
\begin{equation}
\label{eq22}
\mathcal{K}_n[\rho](c_1,c_2) = \inf_{S \in
\mathcal{T}_n}\frac{1}{2^n}\sum_{j=1}^{2^n} \rho(x_{1,j}, S x_{2,j}).
\end{equation}
We note that the sequence $\mathcal{K}_n[\rho]$, $n \geq 0$, of semimetrics can
be constructed iteratively as follows. Let $\mathcal{K}_0[\rho]=\rho$ be the
semimetric on~$X = X/ \varsigma_0$. When $n \geq 0$ every point of the set $X
/\varsigma_{n+1}$ is an unordered pair of points in~$X /\varsigma_{n}$, and we can
assign to this pair the semi-sum of the delta measures at the points. Thus, the
set $X /\varsigma_{n+1}$ is embedded in the space of measures on~$X
/\varsigma_{n}$. There is a~Kantorovich metric on the space of measures on the
metric space $(X/\varsigma_n, \mathcal{K}_n[\rho])$, and it is this metric
that defines the metric $\mathcal{K}_{n+1}[\rho]$ on~$X /\varsigma_{n+1}$.

\vskip1pt

We note that the semimetric $\mathcal{K}_n[\rho]$ defined on the quotient space
$(X/\varsigma_n, \mu/\varsigma_n)$ can also be treated as a~semimetric on the
original space $(X,\mu)$ (one need only take its composite with the quotient
map). In this case, the resulting semimetric triples are isomorphic, and the
map preserving the measure and the semimetric is the quotient map.

\begin{definition}
\label{d24} 
A~sequence of positive numbers~$h_n$, $n \geq 1$, is called
a~\textit{scaling sequence for a~semimetric~$\rho$ and a~dyadic filtration}
$\varsigma = (\varsigma_n)_{n \geq 0}$ of the space $(X,\mu)$ if the following
asymptotic relation holds for sufficiently small $\varepsilon>\nobreak0$:
$$
\mathbb{H}_\varepsilon(X/\varsigma_n, \mu/\varsigma_n,
\mathcal{K}_n[\rho]) \asymp
h_n, \qquad n \to \infty.
$$
The class of scaling sequences for a~semimetric~$\rho$ and
a~filtration~$\varsigma$ is denoted by~$\mathcal{H}(X,\mu,\varsigma,\rho)$.
\end{definition}

As in the case of the entropy of an action, it turns out that the class
of scaling sequences does not depend on the metric.

\begin{theorem}
\label{t8} 
Let\/ $\varsigma = (\varsigma_n)_{n \geq 0}$ be a~dyadic filtration
on a~Lebesgue space\/ $(X,\mu)$ and~let\/ $\rho_1,\rho_2 \in
\mathcal{A}\mathrm{dm}(X,\mu)$ be two metrics on~$(X,\mu)$. Then
$$
\mathcal{H}(X,\mu,\varsigma,\rho_1)=\mathcal{H}(X,\mu,\varsigma,\rho_2).
$$
\end{theorem}

This theorem motivates the following definition.

\begin{definition}
\label{d25} 
A~sequence of positive numbers~$h_n$, $n \geq 1$, is said to be
a~\textit{scaling sequence of a~dyadic filtration\/} 
$\varsigma = (\varsigma_n)_{n\geq 0}$ of the space $(X,\mu)$ 
if it is scaling for some (and then also for every) 
summable admissible metric~$\rho$ on~$(X,\mu)$.
\end{definition}

\subsection{Independence of the metric of a~scaling sequence of a~filtration}
\label{ss6.2}
In this section we give a~proof of Theorem~\ref{t8} using
the theory of admissible semimetrics. The following lemma is crucial in our
proof.

\begin{lemma}
\label{l13} 
Let\/ $\varsigma = (\varsigma_n)_{n \geq 0}$ be a~dyadic filtration
of a~Lebesgue space $(X,\mu)$ and let\/ $n \geq 0$ and\/ $\rho_1,\rho_2 \in
\mathcal{A}\mathrm{dm}(X,\mu)$. Then
$$
\bigl\|\mathcal{K}_n[\rho_1] - \mathcal{K}_n[\rho_2]\bigr\|_m\leq 3 \|\rho_1-\rho_2\|_m.
$$
In other words, the Kantorovich iteration\/~$\mathcal{K}_n$ is a~Lipschitz
function with respect to the\/~$m$-norm with the Lipschitz constant equal to\/~$3$.
\end{lemma}

The proof of Lemma~\ref{l13} uses the following observation which follows
immediately from the triangle inequality.

\begin{statement}
\label{st1}
Let\/ $N \geq 1$, let $\rho$ be a~semimetric on a~set~$X$, and let\/
$x_1,\dots,x_N,\allowbreak y_1,\dots, y_N \in X$. Then
$$
\frac{1}{N}\sum_{j=1}^{N} \rho(x_j,y_j) \leq
\frac{3}{N^2} \sum_{j=1}^{N}\sum_{k=1}^{N} \rho(x_j,y_k).
$$
\end{statement}

\begin{proof}
Let $k,m \in \{0,\dots,N-1\}$. Then for every $j$, \,$1\leq j \leq N$, 
it follows from the triangle inequality that
\begin{equation*}
\label{trin}
\rho(x_j,y_j)\leq \rho(x_j,y_{j+k})+\rho(x_{j+k-m},
y_{j+k})+\rho(x_{j+k-m},y_j).
\end{equation*}
Summing this inequality over all $j$,~$k$,~$m$ modulo~$N$ and dividing 
the result by~$N^3$, we obtain the desired inequality. 
\end{proof}

\begin{proof}[of Lemma~\ref{l13}]
Let $N = 2^n$ and let $c_1,c_2$ be two elements of the partition
$\varsigma_n$. Moreover, let $c_1 = \{x_1,\dots, x_{N}\}$ and $c_2 =
\{y_1,\dots,y_{N}\}$ and let $\rho$ be a~measurable semimetric on~$(X,\mu)$
such that $|\rho_2-\rho_1|\leq \rho$. Then
$$
\mathcal{K}_n[\rho_2](c_1,c_2) \leq \mathcal{K}_n[\rho_1+\rho](c_1,c_2)\leq
\mathcal{K}_n[\rho_1](c_1,c_2)+\max_{s
\in S_{N}}\frac{1}{N}\sum_{j=1}^{N}\rho(x_j,y_{s(j)}),
$$
where the maximum is taken over all permutations~$s$ in the symmetric group
$S_{N}$. Estimating the last term using Assertion~\ref{st1}, we obtain the
inequality
$$
\mathcal{K}_n[\rho_2](c_1,c_2) \leq
\mathcal{K}_n[\rho_1](c_1,c_2)+3\widetilde\rho(c_1,c_2), \quad\!\!
\text{where}\quad\!\! \widetilde\rho(c_1,c_2)\,{=}\,
\frac{1}{N^2}\sum_{j=1}^{N}\sum_{k=1}^{N}\rho(x_j,y_{k}).
$$
Obviously, the symmetric inequality also holds:
$$
\mathcal{K}_n[\rho_1](c_1,c_2) \leq
\mathcal{K}_n[\rho_2](c_1,c_2)+3\widetilde\rho(c_1,c_2).
$$
It only remains to note that the function $\widetilde\rho$ defined in this way
is a~measurable semimetric on the space $X/\varsigma_n$ with the quotient
measure $\mu/\varsigma_n$, and
$\|\widetilde\rho\|_{_{L^1(X^2,\mu^2)}}=\|\rho\|_{_{L^1(X^2,\mu^2)}}$. Hence,
$$
\bigl\|\mathcal{K}_n[\rho_1]-
\mathcal{K}_n[\rho_2]\bigr\|_m\leq 3\|\widetilde\rho\|_{_{L^1(X^2,\mu^2)}} =
3 \|\rho\|_{_{L^1(X^2,\mu^2)}}.
$$
Passing to the infimum over all possible semimetrics~$\rho$ dominating the
modulus of~the difference $|\rho_1-\rho_2|$, we obtain the desired inequality. 
\end{proof}

The following lemma is an analogue of Lemma~9 in~\cite{13}.

\begin{lemma}
\label{l14} 
Let\/ $\rho, \widetilde\rho \in \mathcal{A}\mathrm{dm}(X,\mu)$, 
where\/ $\rho$ is a~metric. Then for every\/ $\varepsilon>\nobreak0$ 
there are\/ $\varepsilon_1>\nobreak0$ and\/ $C_1>\nobreak0$ 
such that the following inequality holds for every\/ $n \geq 0$:
\begin{equation}
\label{eq23}
\mathbb{H}_\varepsilon(X,\mu, \mathcal{K}_n[\widetilde\rho\,])\leq
C_1 \mathbb{H}_{\varepsilon_1}(X,\mu,
\mathcal{K}_n[\rho]).
\end{equation}
\end{lemma}

\vskip1pt

\begin{proof}
Consider the set $\mathcal{M}_\rho\subset \mathcal{A}\mathrm{dm}
(X,\mu)$ consisting of the semimetrics $\widetilde\rho$ for which the
conclusion of the lemma holds. Our objective is to prove that 
$\mathcal{M}_\rho = \mathcal{A}\mathrm{dm}(X,\mu)$.

\vskip1pt

The set $\mathcal{M}_\rho$ obviously contains, together with every semimetric,
all semimetrics dominated by that semimetric, 
that is, if~$\rho_1 \in \mathcal{M}_\rho$ and $\rho_2 \leq \rho_1$, 
then $\rho_2 \in \mathcal{M}_\rho$.

\vskip1pt

We claim that $\mathcal{M}_\rho$ is closed under 
the $m$-norm. Indeed, if a~semimetric $\rho_1$ is contained in the closure
of~$\mathcal{M}_\rho$, then for every $\varepsilon>\nobreak0$ there is a~semimetric
$\rho_2 \in \mathcal{M}_\rho$ for which $\|\rho_1-\rho_2\|_m<\varepsilon^2/96$.
It follows from Lemma~\ref{l13} that the inequality
$$
\bigl\|\mathcal{K}_n[\rho_1]-\mathcal{K}_n[\rho_2]\bigr\|_m<\frac{\varepsilon^2}{32}
$$
holds for every $n \geq 0$. By Lemma~\ref{l10}, the following inequality holds:
$$
\mathbb{H}_\varepsilon(X,\mu,\mathcal{K}_n[\rho_1])\leq
\mathbb{H}_{\varepsilon/4}(X,\mu,\mathcal{K}_n[\rho_2]),
$$
which implies that $\rho_1$ also belongs to~$\mathcal{M}_\rho$.

\vskip1pt

Let $f\colon (X,\mu) \to \mathbb{R}$ be a~measurable function. We denote
by~$d[f]$ the semimetric constructed from $f$ as follows:
$$
d[f](x,y)=|f(x)-f(y)|, \qquad x,y\in X.
$$
If~$f$ is the characteristic function of a~measurable set $A \subset X$,
then $d[f]$ is called a~cut semimetric (or simply a~cut).

It can readily be seen that, if~$f$ is a~Lipschitz function with respect 
to~$\rho$, then the semimetric $d[f]$ belongs to~$\mathcal{M}_\rho$. 
Indeed, if~$|f(x)-f(y)|\leq C\rho(x,y)$ for some constant
$C>\nobreak1$ and all $x,y \in X$, then $\mathcal{K}_n[d[f]]\leq C\mathcal{K}_n[\rho]$
for $n \geq 0$, and therefore
$$
\mathbb{H}_\varepsilon(X,\mu,\mathcal{K}_n[d[f]])\leq
\mathbb{H}_{\varepsilon/C}(X,\mu,\mathcal{K}_n[\rho]),
$$
which implies that $d[f]\in \mathcal{M}_\rho$.

\vskip1pt

We can readily verify the inequality
$$
|d[f_1](x,y)-d[f_2](x,y)|\leq |f_1(x)-f_2(x)|+|f_1(y)-f_2(y)|,
\qquad x,y \in X,
$$
and the function on the right-hand side is a~semimetric on~$X$. This implies
the inequality
$$
\|d[f_1]-d[f_2]\|_m\leq 2\|f_1-f_2\|_{_{L^1(X,\mu)}}.
$$
Let $f \in L^1(X,\mu)$. Approximating the function $f$ by functions Lipschitz
with respect to~$\rho$ in the $L^1$-norm and applying the closedness
of~$\mathcal{M}_\rho$, we see that $d[f]\in \mathcal{M}_\rho$.

It follows from what has been said that $\mathcal{M}_\rho$ contains
all semimetrics that are dominated by a~finite sum of cut semimetrics.
Semimetrics of this kind approximate, with respect to the $m$-norm, an
arbitrary summable admissible semimetric (see~the proof of Lemma~9
in~~\cite{13}). Thus, it follows from the closedness of~$\mathcal{M}_\rho$ 
under the~$m$-norm that $\mathcal{M}_\rho=\mathcal{A}\mathrm{dm}(X,\mu)$. 
\end{proof}

\begin{proof}[of Theorem~\ref{t8}]
By symmetry, it suffices to prove the inclusion $\mathcal{H}(X,\mu,\varsigma,
\rho_1)\subset \mathcal{H}(X,\mu,\varsigma,\rho_2)$. If the first set is empty,
then there is nothing to prove. Suppose that~it~is non-empty and that a~sequence
$h=(h_n)_{n\geq 0}$ of positive numbers is such~that 
$h \in \mathcal{H}(X,\mu,\varsigma,\rho_1)$. 
We claim that $h \in \mathcal{H}(X,\mu,\varsigma,\rho_2)$.

\vskip1pt

It obviously follows from Lemma~\ref{l14} applied to~$\rho\,{=}\,\rho_1$ and
$\widetilde\rho\,{=}\,\rho_2$ that for every $\varepsilon>\nobreak0$ 
we have the relation
$$
\limsup_{n \to \infty} \frac{\mathbb{H}_\varepsilon(X, \mu,
\mathcal{K}_n[\rho_2])}{h_n}<+\infty.
$$
It remains to be proved that the lower limit is bounded away from zero. 
To this end, we apply Lemma~\ref{l14}, transposing the roles of the metrics
$\rho_1$ and $\rho_2$. Namely, when  
$\rho\,{=}\,\rho_2$ and $\widetilde\rho\,{=}\,\rho_1$
and for every $\varepsilon>\nobreak0$ there are an $\varepsilon_1>\nobreak0$ 
and a~$C_1>\nobreak0$ such that
$$
\mathbb{H}_\varepsilon(X,\mu,
\mathcal{K}_n[\rho_1])\leq C_1 \mathbb{H}_{\varepsilon_1}(X,\mu,
\mathcal{K}_n[\rho_2]).
$$
If $\varepsilon$ is sufficiently small, then
$$
0<\liminf_{n \to \infty} \frac{\mathbb{H}_\varepsilon(X,\mu,
\mathcal{K}_n[\rho_1])}{h_n}\leq C_1\liminf_{n \to \infty}
\frac{\mathbb{H}_{\varepsilon_1}(X,\mu,
\mathcal{K}_n[\rho_2])}{h_n}.
$$
Thus, if~$\widetilde\varepsilon<\varepsilon_1$, then
$$
0<\liminf_{n \to \infty} \frac{\mathbb{H}_{\widetilde\varepsilon}(X,\mu,
\mathcal{K}_n[\rho_2])}{h_n}.
$$
This shows that $h \in \mathcal{H}(X,\mu,\varsigma,\rho_2)$ and
completes the proof of the theorem. 
\end{proof}

The following remark enables us to find a~scaling sequence of a~filtration
by making the calculations for a~suitable family of semimetrics.

\begin{statement}
\label{st2}
Let a~sequence of summable admissible semimetrics\/ $(\rho_k)_{k \geq 0}$ 
be non-\allowbreak decreasing and together separate the points of a~space $(X,\mu)$ 
up~to a~set of measure zero. Let\/ $\varsigma = (\varsigma_n)_{n \geq 0}$ be 
a~dyadic filtration on~$(X,\mu)$. Suppose that a~sequence\/ $h=(h_n)_{n \geq 0}$
of positive numbers is scaling for~$\varsigma$ and for every
semimetric~$\rho_k$, $k \geq\nobreak 0$. Then $h$ is scaling for~$\varsigma$.
\end{statement}

\begin{proof}
We first choose small positive coefficients $C_k$, $k \geq\nobreak 0$, in such
a~way that the function $\rho=\sum_{k\geq 0} C_k\rho_k$ is a~summable admissible
metric. By Theorem~\ref{t8} and Definition~\ref{d25}, it suffices to show
that $h \in \mathcal{H}(X,\mu,\varsigma,\rho)$. The metric~$\rho$ obviously dominates
every semimetric~$\rho_k$, $k\geq 0$, and therefore the bound for the
lower limit is obvious. We now estimate the upper limit.

\vskip1pt

Let $\varepsilon>\nobreak0$ and take an index $l$ for which
$$
\biggl\|\rho - \sum_{k= 0}^{l} C_k\rho_k\biggr\|_m = \biggl\|\sum_{k\geq
l+1} C_k\rho_k\biggr\|_{_{L^1(X^2,\mu^2)}}<\frac{\varepsilon^2}{96}.
$$
Then, by Lemma~\ref{l13}, for every $n \geq 0$ we have the inequality
$$
\biggl\|\mathcal{K}_n[\rho] - \mathcal{K}_n\biggl[\sum_{k= 0}^{l}
C_k\rho_k\biggr]\biggr\|_m<\frac{\varepsilon^2}{32},
$$
whence, by Lemma~\ref{l10}, the following inequality holds:
$$
\mathbb{H}_\varepsilon(X,\mu,\mathcal{K}_n[\rho])\,{\leq}\,
\mathbb{H}_{\varepsilon/4}\biggl(X,\mu,\mathcal{K}_n\biggl[\sum_{k= 0}^{l}
C_k\rho_k\biggr]\biggr)\,{\leq}\,
\mathbb{H}_{\varepsilon/4}\biggl(X,\mu,\sum_{k= 0}^{l} C_k \cdot
\mathcal{K}_n[\rho_l]\biggr) \leq C h_n
$$
for some constant $C>\nobreak0$. This readily implies the desired bound 
for the upper limit. 
\end{proof}

For completeness we point out that the definition of the entropy of a~homogeneous
filtration (in particular, a~dyadic one) and the original papers (such as~\cite{9})
used the metric entropies of some partitions of the orbit space
of the automorphism group of a~tree related to a~filtration rather than the
$\varepsilon$-entropy of the metrics. In comparable terms,
the difference is that Kantorovich iterations of semimetrics determined
by arbitrary functions with finitely many values were used. In essence, it was
proved above that the supremum of the scaling sequences of entropies over all
such semimetrics (that is, over all functions with finitely many values) can be
replaced by an arbitrary metric. It seems that such a~result holds for many
inhomogeneous filtrations, for example, for semi-homogeneous ones (that~is, for
filtrations with respect to central measures on the path spaces of arbitrary
graded graphs).

\subsection{Computation of scaling sequences of actions}
\label{ss6.3}
As already mentioned above (see Theorem~\ref{t6}), the class
$\mathcal{H}(\mathcal{T}(\mathrm{OP}),\mu^\sigma,\mathcal{A})$ was computed
in~\cite{27}. In the~same paper, the class
$\mathcal{H}(\mathcal{T}(\mathrm{OP}),\mu^\sigma,\mathcal{D})$ was computed
for the action $\kappa$ of~the group $\mathcal{D}$ with the natural equipment
$\mathcal{D}_n$, $n\geq 1$, and it was shown that these classes coincide. 
In this subsection, we describe another way of computing the class
$\mathcal{H}(\mathcal{T}(\mathrm{OP}),\mu^\sigma,\mathcal{D})$.

\vskip2pt

We note the following general fact.

\begin{lemma}
\label{l15} 
Let an action of some equipped group\/ $G$ on a~space $(X_1,\mu_1)$
have a~scaling sequence\/ $h_n$, and its action on a~space $(X_2,\mu_2)$ have 
a~\textup{`}discrete spectrum\/\textup{'} (that~is, an invariant summable
admissible metric). Then the direct product of the actions on the space
$(X_1\times X_2, \mu_1\times \mu_2)$ also has a~scaling sequence\/ $h_n$.
\end{lemma}

\begin{proof}
Let $\rho_1$ be a~generating admissible semimetric on~$(X_1,\mu_1)$
and $\rho_2$ an invariant admissible metric on~$(X_2,\mu_2)$.
Then the semimetric
$$
\rho\bigl((x_1,x_2),(y_1,y_2)\bigr)=\rho_1(x_1,y_1)+\rho_2(x_2,y_2)
$$
is admissible and generating for the direct product of the actions on the space
$(X_1\times X_2, \mu_1\times \mu_2)$. Here the $\varepsilon$-entropies
of finite averagings of this semimetric admit two-sided estimates, because
of the inequality in Lemma~\ref{l10}, using the $\varepsilon$-entropies of an
averaging of the semimetric $\rho_1$, since the metric $\rho_2$ is invariant
and has finite $\varepsilon$-entropies. 
\end{proof}

This implies an assertion, which was proved in~\cite{27}.

\begin{corollary}
\label{c7} 
The sequence\/ $h_n = 2^{\sum_{i=0}^{n-1}\sigma_i}$ is scaling for the
action\/ $\operatorname{diag}$ of the group\/ $\mathcal{D}$ on the space\/
$(I^{\mathcal{D}}\times I^{\mathbb{N}}, \omega^\sigma)$, and therefore 
for the isomorphic canonical action\/ $\kappa$ of\/~$\mathcal{D}$
on\/~$(\mathcal{T}(\mathrm{OP}),\mu^\sigma)$.
\end{corollary}

\begin{proof}
The action of~$\mathcal{D}$ on~$(I^{\mathbb{N}},m)$ has an invariant metric,
and therefore we arrive at the conditions of Lemma~\ref{l15}. Thus, the action
of~$\mathcal{D}$ on~$(I^{\mathcal{D}}\times I^{\mathbb{N}}, \omega^\sigma)$ has
the same scaling sequence as the action of~$\mathcal{D}$
on~$(I^{\mathcal{D}}, m^\sigma)$. 
The measure~$m^\sigma$ is the pushforward of the Lebesgue measure
on~$I^{\mathcal{D}/\mathcal{D}^\sigma}$ (see Definition~\ref{d2}), and
therefore the scaling sequence of the action on~$\bigl(I^{\mathcal{D}},
m^\sigma\bigr)$ coincides with the scaling sequence of~the~action
on~$\bigl(I^{\mathcal{D}}, m^\sigma\bigr)$ with the Lebesgue measure. The
subgroup $\mathcal{D}^\sigma$ acts trivially on this space, and therefore the
calculation reduces to that of a~scaling sequence of the
`effectively' acting part of the group $\mathcal{D}$, that~is, of the
complementary subgroup
$$
\overline{\mathcal{D}}^{\,\sigma} = 
\bigl\langle g_i \colon \sigma_i=1, \, i \geq 0 \bigr\rangle.
$$
Let $\rho$ be a~cut along the first coordinate
on~$I^{\overline{\mathcal{D}}^\sigma} =I^{\mathcal{D}/\mathcal{D}^\sigma}$, which
is obviously an admissible semimetric that is generating for the action $\mathcal{D}$.
Then the semimetric space $(I^{\overline{\mathcal{D}}^\sigma},
T^{\mathcal{D}_n}_{\mathrm{av}}\rho)$ with the Lebesgue measure is isomorphic
to the dyadic cube of dimension $|\mathcal{D}_n/
(\mathcal{D}^\sigma\cap\mathcal{D}_n)| = 2^{\sum_{i=0}^{n-1}\sigma_i}$ with
the uniform measure, which implies the result. 
\end{proof}

Arguing in a~similar way one can obtain the following result.

\begin{lemma}
\label{l16} 
Let\/ $H \subset \mathcal{D}$ be a~subgroup. Consider the
measure\/ $m^H$ on\/~$I^{\mathcal{D}}$ which is the pushforward  
of the Lebesgue
measure on~$I^{\mathcal{D}/H}$. Then the sequence $h_n= |\mathcal{D}_n/(H\cap
\mathcal{D}_n)|$ is scaling for the direct product of the actions
of\/~$\mathcal{D}$ (with the equipment\/ $\mathcal{D}_n$, $n \geq 1$) 
on the space\/ $(I^{\mathcal{D}}\times I^{\mathbb{N}}, m^H\times m)$.
\end{lemma}

\subsection{Computing a~scaling sequence of a~filtration}
\label{ss6.4}
In this subsection we compute a~scaling sequence
of the filtration $\zeta$ on~$(I^{\mathcal{D}}\times I^{\mathbb{N}},
\omega^\sigma)$, and hence of the (isomorphic to~$\zeta$) tail filtration $\xi$
on the path space $\mathcal{T}(\mathrm{OP})$ of the graph~$\mathrm{OP}$
of ordered pairs with the measure~$\mu^\sigma$.

\vskip1pt

We first prove a~lemma. Let $m \in \mathbb{Z}$, $m \geq 0$. We note that
there is a~dyadic hierarchy on the subgroup $\mathcal{D}_m$, that
of the cosets of the subgroups~$\mathcal{D}_j$, $0< j <\nobreak m$. The bijections
of the set~$\mathcal{D}_m$ that preserve this hierarchy form a~group isomorphic
to the group~$\mathcal{T}_m$ of automorphisms of a~binary tree of height~$m$.
For an arbitrary finite set~$Q$, the action of~$\mathcal{T}_m$
on~$Q^{\mathcal{D}_m}$ by permuting the coordinates is trivially generated.
This action obviously preserves the normalized Hamming metric~$d_H$
on~$Q^{\mathcal{D}_m}$. Let
$$
\operatorname{dist}_m(w^{(1)},w^{(2)}) = \min_{S \in T_m} d_H(w^{(1)},S w^{(2)}),
\qquad w^{(1)},w^{(2)} \in Q^{\mathcal{D}_m},
$$
be the semimetric on~$Q^{\mathcal{D}_m}$ measuring the distance between orbits
under the action of~$\mathcal{T}_m$.

\begin{lemma}
\label{l17} 
Let\/ $Q$ be a~finite set. Let\/ $r,m \in \mathbb{Z}$, \,$0 \leq r\leq m$. 
Let\/ $G_m$ be the subgroup of\/~$\mathcal{D}_m$ generated by some\/ $r$ elements
of\/~$\{g_0, \dots, g_{m-1}\}$. Let
$$
W = \{w \in Q^{\mathcal{D}_m}\colon w(g+\,\cdot\,) = w(\,\cdot\,) \ \forall
g \in G_m\}
$$
be the set of functions on\/~$\mathcal{D}_m$, with values in\/~$Q$, that are
invariant under the shift of the argument by elements of~$G_m$.
Let\/ $\nu$ be the uniform measure on\/~$W$. Then the\/ $\varepsilon$-entropy 
of the triple\/ $(Q^{\mathcal{D}_m},\operatorname{dist}_m,\nu)$ admits a~two-sided
estimate for any sufficiently small positive\/ $\varepsilon$ by\/~$2^{m-r}$ with
multiplicative constants depending on\/~$\varepsilon$ and\/ $|Q|$.
\end{lemma}

\begin{proof}
The measure $\nu$ is concentrated on~$W$, which is the set of configurations
that are invariant under shifts of the argument by elements of~$G_m$. 
The set~$W$ is not invariant under the action of~$\mathcal{T}_m$, 
and therefore the orbit of a~function $w\in\nobreak W$ under the
action of~$\mathcal{T}_m$ can contain functions which are not in~$W$.
Nevertheless, the minimum of the distances in the Hamming metric between such
orbits is achieved at elements of~$W$. Thus, we can pass to the quotient by the
action of~$G_m$, under which the triple $(Q^{\mathcal{D}_m},\operatorname{dist}_m,\nu)$
passes to an isomorphic triple. Hence, the $\varepsilon$-entropy of this
triple is a~function of~$m-r$. We denote the $\varepsilon$-entropy under 
consideration by~$h_{m-r}(\varepsilon)=h_{m-r}(\varepsilon, |Q|)$.

\vskip1pt

In what follows, we may assume that $r=\nobreak0$. Then $W=Q^{\mathcal{D}_m}$. 
We seek two-\allowbreak sided bounds for the number~$h_m(\varepsilon)$. The upper 
bound reduces to the standard entropy estimate. Indeed, this number does not 
exceed the $\varepsilon$-entropy of the uniform measure on the metric space
$(Q^{\mathcal{D}_m}, d_H)$. This space is a~hypercube (with side~$|Q|$)
of dimension~$2^m$ whose $\varepsilon$-entropy is well known and asymptotically
proportional to the dimension of the cube, that is, to~$2^m$.

The lower bound is found in a~somewhat more subtle way and requires an analysis
of the cases $|Q|=2$ and $|Q|>2$. In both cases, we estimate the size $M_m$
of a~maximal orbit in the space $Q^{\mathcal{D}_m}$ under the action of~$\mathcal{T}_m$. 
We shall find a~constant $C_1=C_1(\varepsilon,|Q|)$ such that
the ball of radius~$\varepsilon$ in the metric space $(Q^{\mathcal{D}_m}, d_H)$
contains at most $C_1(\varepsilon,|Q|)^{2^m}$ points, and $C_1(\varepsilon,|Q|)
\to 1$ as~$\varepsilon \to 0$ for a~fixed $|Q|$. The $\varepsilon$-neighborhood
of every orbit contains at most $C_1(\varepsilon,|Q|)^{2^m}M_m$ points, and
therefore the $\varepsilon$-entropy of the uniform measure admits the
lower bound
\begin{equation}
\label{eq24}
h_m(\varepsilon) \geq
\log\biggl((1-\varepsilon)\frac{|Q|^{2^m}}{C_1(\varepsilon,|Q|)^{2^{m}}M_m }\biggr).
\end{equation}

When $|Q|>2$ it suffices to note that $M_m\leq |\mathcal{T}_m| = 2^{2^{m}-1}$.
Therefore, by inequality~\eqref{eq24} we have
$$
h_m(\varepsilon) \geq2^{m} \log\biggl(\frac{|Q|}{2
C_1(\varepsilon,|Q|)}\biggr)+\log(1-\varepsilon)\geq 2^m C_2(\varepsilon,|Q|)
$$
for sufficiently small $\varepsilon$ and some positive constant
$C_2(\varepsilon,|Q|)$.

\vskip1pt

When~$|Q|=2$, the number $M_m$ can be estimated by induction. Obviously, the
inequality $M_{m+1}\leq 2 M_m^2$ holds and 
we have $M_2 = 4$, and
therefore $M_m \leq\nobreak 2^{3\cdot 2^{m-2}-1}$\!. Applying this inequality 
together with~\eqref{eq24}, we obtain the bound
$$
h_m(\varepsilon) \geq 2^{m}
\log\biggl(\frac{2^{1/4}}{C_1(\varepsilon,2)}\biggr)+\log(1-\varepsilon)\geq
2^m C_2(\varepsilon)
$$
for sufficiently small $\varepsilon$ and some positive constant
$C_2(\varepsilon)$.

Summarizing what has been said, we obtain that 
$h_m(\varepsilon) \asymp 2^{m}$, \,$m\to\nobreak\infty$. 
\end{proof}

\begin{theorem}
\label{t9}
The sequence\/ $h=(h_n)$, \,$h_n=2^{\sum_{i=0}^{n-1} \sigma_i}$, \,$n \geq 1$, is
a~scaling sequence of the filtration\/ $\zeta = (\zeta_n)_{n \geq 0}$ on the space\/
$I^{\mathcal{D}}\times I^{\mathbb{N}}$ with the measure\/ $\omega^\sigma$ and
also~of the filtration\/ $\xi=(\xi_n)_{n\geq 0}$ on the space\/
$\mathcal{T}(\mathrm{OP})$ with the measure $\mu^\sigma$, which is isomorphic
to~$\zeta$.
\end{theorem}

\begin{proof}
Let us compute the asymptotic behaviour of the entropies for the following
sequence of semimetrics $\rho_k$, $k\geq 1$, on~$I^{\mathcal{D}}\times
I^{\mathbb{N}}$. Let $w^{(1)},w^{(2)} \in I^{\mathcal{D}}$ and
$\alpha^{(1)},\alpha^{(2)} \in\nobreak I^{\mathbb{N}}$. We write
\begin{align*}
&\rho_k\bigl((w^{(1)}, \alpha^{(1)}),(w^{(2)}, \alpha^{(2)})\bigr)
\\[4pt]
&\qquad=\begin{cases} 0& \text{ if }
w^{(1)}|_{\mathcal{D}_k}=w^{(2)}|_{\mathcal{D}_k} \text{ and }
\alpha^{(1)}_i=\alpha^{(2)}_i \text{ for } i=1,\dots, k,
\\
1& \text{ otherwise.}
\end{cases}
\end{align*}

It is clear that this sequence of semimetrics is monotone increasing and
separates the points of the space by~$I^{\mathcal{D}}\times
I^{\mathbb{N}}$. By Assertion~\ref{st2}, it suffices to show that the sequence
$h$ is scaling for every semimetric $\rho_k$, $k \geq 1$.

\vskip1.5pt

Let $k$ be a~fixed positive integer and let $n>k$. Our objective is to understand
the structure of the semimetric $\mathcal{K}_n[\rho_k]$ on the set of elements
of the partition~$\zeta_n$.

\vskip1.5pt

We define a~map~$\phi_n \colon I^{\mathcal{D}}\times I^{\mathbb{N}}\to
I^{\mathcal{D}_n}$ as follows. Every element of the partition $\zeta_n$ is the
orbit of some point under the action of the group $\mathcal{D}_n$, and this
element contains a~unique pair $(w,\alpha)$, $w \in I^{\mathcal{D}}$, $\alpha
\in I^{\mathbb{N}}$, for which $\alpha_i=\nobreak0$ for $i=1,\dots, n$. 
We say that this pair is the \textit{representative} of this element of 
the partition and set $\phi_n=w|_{\mathcal{D}_n}$ on this orbit.

Let $(w^{(1)},\alpha^{(1)})$ and $(w^{(2)},\alpha^{(2)})$ be the representatives
of elements $c_1$ and $c_2$, respectively, of the partition $\zeta_n$. Then
by~\eqref{eq22} we have (recalling that the group~$\mathcal{T}_n$ acts
on~$\mathcal{D}_n$ and preserves the hierarchy and that~$\tau$ embeds~$\mathcal{D}$
in~$I^{\mathbb{N}}$)
\begin{align}
\mathcal{K}_n[\rho_k](c_1,c_2)
&=\min_{S\in \mathcal{T}_n}\bigg\{\frac{1}{2^n}\sum_{g\in
\mathcal{D}_n}\rho_k\Bigl(\bigl(w^{(1)}(g+\,\cdot\,),
\tau(g)+\alpha^{(1)}\bigr),
\nonumber
\\
&\qquad\qquad\qquad\bigl(w^{(2)}(Sg+\,\cdot\,),
\tau(Sg)+\alpha^{(2)}\bigr)\Bigr) \bigg\}.
\label{eq25}
\end{align}

Let $\mathcal{D}_{n,k}$ be the subgroup of~$\mathcal{D}_n$ generated by the
elements $g_{k+1},\dots, g_n$, which is complementary to~$\mathcal{D}_k$.
The minimum in expression~\eqref{eq25} is obviously achieved on the
transformations $S \in \mathcal{T}_n$ for which $Sg-g \in \mathcal{D}_{n,k}$
for every $g \in \mathcal{D}_n$. Transformations of this kind can be represented
in the form $S(g+\widetilde g)=g+\widetilde S \widetilde g$, where $g \in
\mathcal{D}_k$, \,$\widetilde g \in \mathcal{D}_{n,k}$, and the transformation
$\widetilde S$ acts on~$\mathcal{D}_{n,k}$, preserving the hierarchy.
Therefore, we can rewrite formula~\eqref{eq25} in the form
\begin{align}
\mathcal{K}_n[\rho_k](c_1,c_2)
&=\min_{\widetilde S}\biggl\{\frac{1}{2^n}
\sum_{\substack{g\in \mathcal{D}_k\\ \widetilde g \in \mathcal{D}_{n,k}}}
\rho_k\Bigl(\bigl(w^{(1)}(g+\widetilde g+\,\cdot\,),\tau(g+\widetilde
g)+\alpha^{(1)}\bigr),
\nonumber
\\
&\qquad\qquad
\bigl(w^{(2)}(g+\widetilde S\widetilde
g+\,\cdot\,),\tau(g+\widetilde S \widetilde
g)+\alpha^{(2)}\bigr)\Bigr)\biggr\}
\nonumber
\\[3pt]
&=\min_{\widetilde S} \biggl\{\frac{1}{2^{n-k}}\Bigl|\bigl\{\widetilde g \in
\mathcal{D}_{n,k} \colon w^{(1)}(\widetilde
g +\,\cdot\,)\big|_{\mathcal{D}_k}\ne
w^{(2)}(\widetilde S \widetilde
g +\,\cdot\,)\big|_{\mathcal{D}_k}\bigr\}\Bigr| \bigg\}.
\label{eq26}
\end{align}

The restriction of a~configuration $w \in I^{\mathcal{D}}$ to~$\mathcal{D}_n$
can be viewed as an element of the space $Q^{\mathcal{D}_{n,k}}_k$ with
$Q_k=I^{\mathcal{D}_k}$. We equip $Q_k^{\mathcal{D}_{n,k}}$ with the
Hamming metric~$d_H$. The group $\mathcal{T}_{n-k}$ acts on this space,
preserving the hierarchy. The last expression in formula~\eqref{eq26} is
precisely the distance between the orbits of this action that contain
$w^{(1)}|_{\mathcal{D}_n}=\phi_n(w^{(1)},\alpha^{(1)})$ and
$w^{(2)}|_{\mathcal{D}_n}=\phi_n(w^{(2)},\alpha^{(2)})$. In other words, 
the~map~$\phi_n$ is an isometry taking the semimetric space
$\bigl((I^{\mathcal{D}}\times I^{\mathbb{N}})/\zeta_n,
\mathcal{K}_n[\rho_k]\bigr)$ to~the metric space
$(Q_k^{\mathcal{D}_{n,k}},d_H)$.

\vskip1pt

Let $\nu_n$ be the pushforward  
of the measure $\omega^\sigma$ under the map
$\phi_n$. This is a~measure on~$I^{\mathcal{D}_n}=Q_k^{\mathcal{D}_{n,k}}$
concentrated and uniform on the set of configurations $w \in I^{\mathcal{D}_n}$
that satisfy the relation $w(g_i+\,\cdot\,)=w(\,\cdot\,)$ for $0\leq i\leq
n-1$, \,$\sigma_i=\nobreak0$. Let $Q = \{w \in I^{\mathcal{D}_k}\colon
w(\,\cdot\,+g_i)=w(\,\cdot\,) \text{ for } 0 \leq i\leq k-1,\, \sigma_i=0\}$.

\vskip1.5pt

Take $m=n-k$. The group $\mathcal{D}_{n,k}$ is isomorphic to the group
$\mathcal{D}_m$ (by a~shift of the indexing of the generators). Let $G_m$ be the
subgroup of~$\mathcal{D}_{n,k}$ generated by the elements $g_i$, $k\leq i \leq
n$, for which $\sigma_i=\nobreak0$. We arrive at the conditions of Lemma~\ref{l17},
where $\nu_n$ is the pushforward of~$\omega^\sigma$
under~$\phi_n$ and~$W$ is its support. Applying the lemma,
we find the asymptotic behaviour of the $\varepsilon$-entropy of the quotient
of~$\omega^\sigma$ with respect to the partition~$\zeta_n$ on the
space $(I^{\mathcal{D}}\times I^{\mathbb{N}})/\zeta_n$ with the
semimetric~$\mathcal{K}_n[\rho_k]$ constructed from the semimetric $\rho_k$:
$$
\mathbb{H}_\varepsilon\bigl((I^{\mathcal{D}}\times I^{\mathbb{N}})/\zeta_n,
\omega^\sigma / \zeta_n, \mathcal{K}_n[\rho_k]\bigr) =
\mathbb{H}_\varepsilon\bigl(Q_k^{\mathcal{D}_{n,k}},
\nu_n,d_H\bigr) \asymp 2^{\sum_{i=k}^n \sigma_i} \asymp h_{n}, \quad
n\to\infty.
$$

Thus, we have proved that the sequence $h$ is scaling for the semimetric
$\rho_k$, $k \geq\nobreak 1$, and for the filtration $\zeta$ on the space
$(I^{\mathcal{D}}\times I^{\mathbb{N}}, \omega^\sigma)$. By
Assertion~\ref{st2}, this means that $h$ is a~scaling sequence of the
filtration $\zeta$. This completes the proof of the theorem. 
\end{proof}

\section{Conclusion}
\label{s7}

\vskip1pt

1.~The universal model of an adic action of the group $\mathbb Z$ 
(and also of the infinite sum of groups of order~$2$) on the path space of the 
graph~$\mathrm{OP}$, which is suggested in the paper, can be compared with the
generally accepted and well-known symbolic model of a~group action, which is
also universal. The advantage of the model suggested here is that it already
has a~canonical periodic approximation of the action, which is absent in the
symbolic model.

\vskip1.5pt

2.~The universal model assumes a~description of all central measures on the
path space of the graph. This description can be found in the paper, and it turns
out to be highly visible modulo the description of the invariant measures in the
symbolic model.

\vskip1.5pt

3.~The scale of intermediate arbitrary sublinear asymptotic behaviours of the
scaled entropy (the existence of these asymptotic behaviours has only been discovered
recently) is associated with the scaled entropy of the filtrations on the path
space of the graph, that is, finally, with the rate of periodic approximations
of the adic action of groups. This coincidence reveals a~deep connection between
the theories of filtrations and periodic approximations. This connection will be
investigated in future papers.

\end{document}